\documentclass[12pt,leqno,indentfirst]{article}

\usepackage{amstext}

\usepackage{amsthm}
\usepackage{amsopn}
\usepackage{amsfonts}
\usepackage{amsmath}
\usepackage{amssymb}
\usepackage{amsfonts}
\usepackage{pstricks}

\newtheorem{theorem}{Theorem}[section]
\newtheorem{lemma}[theorem]{Lemma}

\newtheorem{corollary}[theorem]{Corollary}

\theoremstyle{definition}

\newtheorem{example}[theorem]{Example}

\theoremstyle{remark}
\newtheorem{remark}[theorem]{\bf Remark}

\numberwithin{equation}{section}

\def\q{\quad}

\def\Qbar{\text{\sl Q\kern-.45em{\vrule height.63em width.05em
depth-.033em}}~}
\def\qbar{{{\scriptstyle Q}\kern-.45em{\vrule height.41em width.035em
depth-.03em}}~}
\def\Cbar{\text{\sl C\kern-.35em{\vrule height.63em width.05em
depth-.033em}}~}
\def\cbar{{{\scriptstyle C}\kern-.41em{\vrule height.42em width.035em
depth-.03em}}~}
\def\ibid{\hbox to .5truein{\hrulefill}}
\def\IH{\text{{\rm I}\kern-.13em{\rm H}}}

\def\IR{{\Bbb R}}
\def\IT{{\Bbb T}}
\def\pd#1#2{\frac{\partial#1}{\partial#2}}
\def\twoheaddown{\downarrow\kern-0.78em\raise0.25em\hbox{$\downarrow$}}
\def\headtaildown{\downarrow\kern-0.79em\raise 0.5em\hbox{$\ssize\curlyvee$}}

\def\vsk{\vskip.5cm}

\def\noi{\noindent}

\def\pr{\prime}

\def\la{\langle}
\def\ra{\rangle}
\def\wt{\widetilde}
\def\wh{\widehat}

\def\os{\overset}
\def\us{\underset}

\def\CA{{\cal A}}
\def\SD{{\cal D}}

\def\SM{{\cal M}}

\def\CS{{\cal S}}

\def\spa{\text{\rm span}}
\def\supp{\text{\rm supp}}

\def\elra{\hbox to 2in{\rightarrowfill}}
\def\ella{\hbox to 2in{\leftarrowfill}}
\def\hrf{\hbox to 2in{\hrulefill}}
\def\hdotfill{\leaders\hbox to 1em{\hss .\hss}\hfill}

\def\balpha{{\pmb \alpha  }}
\def\bal{{\pmb \alpha  }}

\def\bx{{\pmb x}}
\def\by{{\pmb y}}

\def\TT{{\widetilde T}}

\begin{document}

\title{Convexity, moduli of smoothness and a Jackson-type inequality}

\author{Z. Ditzian and A. Prymak}

\date{} \maketitle

\begin{abstract}
For a Banach space $B$ of functions which satisfies for some $m>0$
$$
\max(\Vert  F+G\Vert  _B,\Vert  F-G\Vert  _B) \ge (\Vert  F\Vert
^s_B + m\Vert  G\Vert  ^s_B)^{1/s}, \q \forall \;F,G\in B \leqno(*)
$$
a significant improvement for lower estimates of the moduli of
smoothness
$\omega  ^r(f,t)_B$ is
achieved.  As a result of these estimates, sharp Jackson
inequalities which are superior to the classical Jackson type
inequality
are derived.  Our investigation covers Banach spaces of
functions on $\IR^d$ or $\IT^d$ for which translations are isometries or
on $S^{d-1}$ for which rotations are isometries.  Results for $C_0$
semigroups of contractions are derived.  As applications of the
technique used in this paper, many new theorems are deduced.  An
$L_p$ space with $1<p<\infty  $ satisfies $(*)$ where $s=\max(p,2),$ and
many Orlicz spaces are shown to satisfy $(*)$ with appropriate $s.$

\vsk
\noi
\begin{tabular}{ll}
\text{\bf Key words and phrases:} &\q\text{\rm Moduli of smoothness, Jackson
inequality,}\\
&\q $L_p $\;\text{\rm spaces, Orlicz spaces.}\\
\\
\text{\bf AMS subject classification:} &\q 41A63, 41A17, 41A25,
46B20, 47D60
\end{tabular}

\end{abstract}

\baselineskip20pt
\mathsurround=3pt

\parindent=25pt

\vsk
\section{Introduction}\label{Sec1}

For a Banach space $B$ of functions on $\IR^d$ or $\IT^d$ for which
translations are continuous isometries and whose norm satisfies for some $1<q\le
2$ and some $M\ge 1$
\begin{equation}\label{Eq1.1}
\frac 12\;\Vert  {F+G}\Vert_B   +\,\frac 12\;\Vert
{F-G}\Vert  _B\le (\Vert  F\Vert  ^q_B + M\Vert
G\Vert  ^q_B)^{1/q}, \; \forall \, F, G\in B,
\end{equation}
the first author (see \cite{Di88}) derived a sharp version of the
Marchaud inequality i.e. an estimate of the $r\text{\rm -th}$
modulus of smoothness
$\omega  ^r(f,t)_B$ (see (\ref{Eq1.5}) below) by an
expression involving $\omega  ^{r+1}(f,t)_B\,,$ which implies a sharper
version of the converse inequality (see also \cite{To}). Analogous
results were achieved for functions on the sphere (see \cite{Di99}).
In the other direction, a sharp Jackson inequality and a sharp lower
estimate of $\omega  ^r(f,t)_{L_p}$ for $1<p<\infty  $ were given
in \cite{Da-Di-Ti} using a version of the Littlewood-Paley
inequality.  Here, we will use the following dual inequality to
(\ref{Eq1.1}), given by
\begin{equation}\label{Eq1.2}
\max \,(\Vert  F+G\Vert  _B,\Vert F  -G\Vert  _B) \ge (\Vert  F\Vert
^s_B + m\Vert  G\Vert  ^s_B)^{1/s},\; \forall\, F,G\in B
\end{equation}
for some $2\le s<\infty  $ and $m>0,$ to obtain the sharp Jackson
inequality and the lower estimate of $\omega  ^r(f,t)_B\,.$
This includes the result for $L_p\,,$\, $1<p<\infty  ,$ since for
$B=L_p$ when $1<p<\infty  , $ (\ref{Eq1.2}) is satisfied with
$s=\max\,(2,p).$  An important portion of the paper will be
dedicated to the lower estimate of $\us{0<u\le t}\sup\,\big\Vert
\big(T(u)-I\big)^rf\big\Vert  _B,$ where $T(u)$ is a $C_0$ semigroup
of contractions, and to applications of the lower estimate in
approximation theory.  An example of such an application is the
sharp Jackson inequality for polynomial approximation on a simplex
with Jacobi weights using the $L_p$ norm where $1<p<\infty  $ or some
other Orlicz norm which satisfies (\ref{Eq1.2}).

The condition (\ref{Eq1.2}) depends on the particular norm of $B$
and may not be satisfied by an equivalent norm of $B.$  For our
results we will need a norm on $B$ which satisfies simultaneously
(\ref{Eq1.2}) and the condition that $T(u)$ is a contraction on $B$
or that translation by $\xi$ is a contraction or an isometry on $B,$
which also is not inherited by an equivalent norm.  However, for the
conclusion of our results any equivalent norm of $B$ will do. In
short, we need the condition that $B$ possesses a norm
 for which $T(u)$ are contractions and which simultaneously
satisfies (\ref{Eq1.2}); however,
the results are valid for any equivalent norm on $B.$

The following theorem is perhaps typical of the results achieved in
the present paper.

\begin{theorem}
Suppose $B$ is a Banach space of functions on $\IR^d$ or $\IT^d$ with a
norm
satisfying {\rm (\ref{Eq1.2})} for some $s,$ \, $2\le s<\infty  $ and
\begin{equation}\label{Eq1.3}
\Vert  f(\cdot\, + \xi  )\Vert  _B = \Vert  f(\cdot)\Vert  _B,\q
\us{\vert  h\vert \to 0 }\lim\, \Vert  f(\cdot\, + h) -
f(\cdot)\Vert  _B = 0, \q \Vert  f(\,-\,\cdot\,)\Vert  _B= \Vert
f(\,\cdot\,)\Vert  _B
\end{equation}
for any $f\in B$ and  $\xi  ,\, h\in \IR^d.$
Then for $C$ independent of $f,t$ and $n$
\begin{equation}\label{Eq1.4}
2^{-nr} \Big\{\sum^n_{j=1} 2^{jrs}\omega  ^{r+1}
(f,2^{-j})^s_B\Big\}^{1/s} \le C\omega  ^r(f,2^{-n})_B
\end{equation}
where
\begin{equation}\label{Eq1.5}
\omega  ^r(f,t)_B =\us{\vert  h\vert  \le t}\sup\,\Vert  \Delta
^r_hf\Vert  _B \,,\,\Delta  _hf(x) = f(x+h)-f(x) \q\text{\rm and} \q
\Delta  ^{\ell+1}_h f=\Delta  _h(\Delta  ^\ell_h f).
\end{equation}
\end{theorem}

The inequality (\ref{Eq1.4}) is sharper than the classical $\omega
^{r+1}(f,t)_B \le 2\omega  ^r(f,t)_B$ and is shown in
\cite[Section~10]{Da-Di-Ti}
to be optimal for $L_p\,,$ \, $1<p<\infty  .$

Throughout this paper constants will be positive and may depend on
the space
\newline $(B,C(\IR^d),L_p(\IR^d)$ etc.) and on $r$ but will be
valid for all the elements of the space and will be independent of
$t,n,j$ and $\ell.$  Furthermore, unless otherwise specified, when a
condition, result, or estimate is given in a theorem, definition, or
remark concerning functions in some space, it applies to all the
functions in that space.

\section{The basic inequality}\label{Sec2}

In this section we derive the basic inequality used throughout this
paper.

\begin{theorem}\label{Thm2.1}
Suppose that $B$ is a Banach space of functions, that $T:B\to B$ is a
linear contraction operator,  that is $\Vert  Tf\Vert  _B\le \Vert
f\Vert  _B$ and suppose also that {\rm (\ref{Eq1.2})} is satisfied
 with a given $s,$ \, $2\le s<\infty  $ and $m>0.$  Then
for some $m_1>0$
\begin{equation}\label{Eq2.1}
\Vert  \Delta  ^r_Tf\Vert  _B\ge m_1\Big(\sum^\infty_{j=0}
2^{-jrs}\Vert  \Delta  ^{r+1}_{T^{2^j}} f\Vert  ^s_B\Big)^{1/s}
\end{equation}
where
\begin{equation}\label{Eq2.2}
\Delta  _{T^{2^\ell}}f = T^{2^\ell}f-f \q \text{\rm and}\q \Delta
^{k+1}_{T^{2^\ell}} f = \Delta  _{T^{2^\ell}} \big(\Delta
^k_{T^{2^\ell}} f\big).
\end{equation}
\end{theorem}

\begin{remark}\label{Rem2.2}
Examples of such Banach spaces on $\IR^d,\, \IT^d$ or $S^{d-1}$ are
$L_p$ spaces for

\noi
$1<p<\infty  $ where $s=\max (p,2).$  An example of $T$
on a space of functions on $\IR^d$ (and $\IT^d)$ is $Tf(x) = f(x+\xi  )$
with $x,\xi  \in \IR^d.$  An example of $T$ on a space of functions
on $S^{d-1}$ is $Tf(x) = f(\rho x)$ with $x\in S^{d-1}$ and $\rho\in
SO(d)$ (the orthogonal matrices on $\IR^d$ whose determinant equals
$1).$
Also $T = T(t)$ may be a semigroup  of contractions, the
simplest being $T(t)f(x) = f(x+t)$ on $L_p(\IR_+),$ but other examples
important for applications will be described at length.
\end{remark}

\begin{proof}
Let $\TT$ be any linear contraction operator on $B$. We note that $\TT^n f= \TT(\TT^{n-1}f)$ and follow
\cite{Di88} to define $F=\frac 12(\TT^2-I)\varphi  $ and $G=-\,\frac 12(\TT-I)^2
\varphi$, $\varphi\in B$, so that $F+G = (\TT-I)\varphi  $ and $F-G = \TT(\TT-I)\varphi$.
As $\TT$ is a contraction, we have $\max(\Vert F+G\Vert  _B,\Vert
F-G\Vert  _B) = \Vert  F+G\Vert  _B = \Vert  (\TT-I)\varphi  \Vert  _B$ and by~\eqref{Eq1.2} with
$\varphi =(\TT-I)^{r-1}f$, $f\in B$, we obtain
\begin{equation}\label{Eq2.3}
\Vert  (\TT-I)^r f\Vert  ^s_B
\ge \frac{1}{2^s}\,\Vert
(\TT^2-I)(\TT-I)^{r-1}f\Vert  ^s_B + m\,\frac{1}{2^s}\,\Vert
(\TT-I)^{r+1}f\Vert ^s_B.
\end{equation}
Recalling that $\TT$ is a contraction, we have
\[
\Vert (\TT^2-I)^r f\Vert ^s_B=
\Vert (\TT+I)^r(\TT-I)^r f\Vert ^s_B \le
2^{(r-1)s} \Vert (\TT+I)(\TT-I)^r f\Vert ^s_B =
2^{(r-1)s} \Vert (\TT^2-I)(\TT-I)^{r-1} f\Vert ^s_B,
\]
which, combined with~\eqref{Eq2.3}, yields
\begin{equation}\label{Eq2.4}
\Vert  (\TT-I)^r f\Vert  ^s_B
\ge \frac{1}{2^{rs}}\,\Vert
(\TT^2-I)^rf\Vert  ^s_B + m\,\frac{1}{2^s}\,\Vert
(\TT-I)^{r+1}f\Vert ^s_B.
\end{equation}
Now we use~\eqref{Eq2.4} iteratively with $\TT=T$, $\TT=T^2$, $\TT=T^4$, $\dots$, $\TT=T^{2^\ell}$ to obtain
$$
\begin{aligned}
\Vert  (T-I)^r f\Vert  ^s_B
&\ge \frac{1}{2^{rs}}\, \Vert  (T^2-I)^r f\Vert  ^s_B +
m\,\frac{1}{2^s}\,\Vert  (T-I)^{r+1}f\Vert  ^s_B\\
&\ge \frac{1}{2^{2rs}} \,\Vert  (T^4-I)^r f\Vert  ^s_B +
m\,\frac{1}{2^s}\,\Big( \Vert  (T-I)^{r+1} f\Vert  ^s_B
+\frac{1}{2^{rs}}\,\Vert  (T^2-I)^{r+1}f\Vert  ^s_B\Big)\\
&\ge \ldots\ge \frac{1}{2^{(\ell+1)rs}}\,\Vert  \big(T^{2^{\ell+1}}-I\big)^r
f\Vert  ^s_B +
m\,\frac{1}{2^s}\,\Big(\sum^{\ell}  _{j=0}\,\frac{1}{2^{rsj}}\,\Vert
(T^{2^j}-I)^{r+1} f\Vert  ^s_B\Big)\\
&\ge
\Big(\frac{m}{2^s}\Big)\,\sum^{\ell}  _{j=0}\,\frac{1}{2^{rsj}}\,\Vert
(T^{2^j}-I)^{r+1}f\Vert  ^s_B,
\end{aligned}
$$
which implies~\eqref{Eq2.1} with $m_1=\displaystyle\frac{m^{1/s}}2$.
\end{proof}

The inequality (\ref{Eq2.1}), which is at the core of most of the results
in this paper, is very simple, but to apply it successfully, we will
need many and perhaps more sophisticated results.

\section{The condition on the space}\label{Sec3}
In this section we will discuss the condition (\ref{Eq1.2}), exhibit
spaces for which it is valid and for what $s.$  The condition
(\ref{Eq1.1}) was shown in \cite{Di88} to be equivalent to the
condition
\begin{equation}\label{Eq3.1}
\eta  _B(\sigma  ) = \us{\os{\Vert  E\Vert  =1}{\us{\Vert  G\Vert
=\sigma}{}}}\sup\,\Big(\frac 12\, \Vert  F+G\Vert  _B +\frac 12\, \Vert
 F-G\Vert  _B-1\Big), \q \eta  _B(\sigma  ) \le k\sigma  ^q,
\end{equation}
which was extensively investigated, and spaces $B$ satisfying
 (\ref{Eq3.1}) are described (see \cite[p.63]{Li-Tz}) as having
modulus of
smoothness of power type $q.$
We note that the concept modulus of smoothness in \cite{Li-Tz}
describes the smoothness of the unit ball of the Banach space $B$
(in relation to a specific norm), and is not related to the concept
with the same name (see for instance (\ref{Eq1.5})) in approximation
theory describing smoothness of a function (i.e. an element of $B).$
We note that we found (\ref{Eq1.1})
easier to use in classical analysis and also  easier to verify (see
\cite[p.49]{De-Lo}).

In the next theorem we show that
(\ref{Eq1.2}) is dual to (\ref{Eq1.1}), and use that later to examine
spaces that satisfy (\ref{Eq1.2}) and for what~$s.$  As a result we
will show (later) that a big class of Orlicz spaces satisfies
(\ref{Eq1.2}) and give examples of such spaces.

\begin{theorem}\label{Thm3.1}
Suppose $B$is a Banach space endowed with a norm which  for
some $q,$\; $1<q\le 2  $ satisfies
\begin{equation}\label{Eq3.2}
\frac 12\,\Vert  {x+y}\Vert  _B + \,\frac 12\,\Vert
{x-y}\Vert  _B\le (\Vert  x\Vert  ^q_B + M\Vert  y\Vert
^q_B)^{1/q} \q\text{\rm for all}\q x,y\in B.
\end{equation}
Then the dual of $B,$ \, $X=B^*$ (with the norm dual to that
satisfying {\rm (\ref{Eq3.2})}) satisfies
\begin{equation}\label{Eq3.3}
\max(\Vert  \varphi  +\psi  \Vert  _{X},\,\Vert  \varphi  -\psi
\Vert  _{X})\ge (\Vert  \varphi  \Vert  ^s_{X} + m\Vert  \psi
\Vert  ^s_{X})^{1/s}\q\text{\rm for all}\q \varphi  ,\psi  \in
X=B^*
\end{equation}
with $s=\frac{q}{q-1}\;\big(\frac 1s +\frac 1q = 1\big)$ and
$m=M^{-1/(q-1)}.$  Moreover, if for a given norm of $X$
{\rm(\ref{Eq3.3})} is satisfied, then $B=X^*$ (with norm dual to that
satisfying {\rm(\ref{Eq3.3})}) satisfies {\rm(\ref{Eq3.2}).}

\end{theorem}

\begin{proof} Define the operator $A$ on $(x,y)\in B\times B=\wt B$
by
$$
A(x,y) =\Big(\frac{x+y}{2}\,,\,\frac{x-y}{2}\Big)
$$
which we consider as a transformation between $\wt B$ with the norm
$\Vert  (u,v)\Vert  _{\wt B_1} = \Vert  u\Vert  _B + \Vert  v\Vert
_B$ and $\wt B$ with the (equivalent) norm $\Vert  (u,v)\Vert  _{\wt
B_2} = (\Vert  u\Vert  ^q_B + M\Vert  v\Vert  ^q_B)^{1/q}.$

Using (\ref{Eq3.2}), we now have   $\Vert  A\Vert  _{\wt B_2\to \wt
B_1}\le 1.$ The dual to $\wt B,\, \wt B\,^*,$ is given by
 $(\varphi  ,\psi  )(u,v) = \varphi  u + \psi
v$ where $\varphi  ,\,\psi \in B^*.$  To calculate $A^*,$ we write
$$
(\bar\varphi  ,\bar\psi  )A(x,y) = \frac 12\,(\bar\varphi  ,\bar\psi
)({x+y},\,{x-y}) = \frac 12\,({\bar \varphi
 +\bar \psi  },\,{\bar\varphi  -\bar\psi  })(x,y
) = A^*(\bar\varphi  ,\bar\psi  )(x,y).
$$
Setting $\frac{\bar\varphi  +\bar \psi  }{2} =\varphi  ,$\,
$\frac{\bar\varphi  -\bar\psi  }{2} = \psi,  $ we have $A^*(\varphi
+\psi  ,\varphi  -\psi  ) = (\varphi  ,\psi  ).$

Since $\Vert  A\Vert  _{\wt B_2\to \wt B_1}\le 1,$ we have $\Vert
A^*\Vert  _{\wt B\,^*_1\to \wt B\,^*_2} \le 1.$  We now write
$$
\Vert  (\varphi  ,\psi  )\Vert  _{\wt B\,^*_1} = \us{\Vert  u\Vert
_B+\Vert  v\Vert  _B=1}\sup\, \vert  (\varphi  u+\psi  v)\vert  \le
\max\,(\Vert  \varphi  \Vert  _{B^*}, \Vert  \psi  \Vert  _{B^*}),
$$
and equality follows,
choosing $v=0$ if $\Vert  \varphi  \Vert
_{B^*}\ge \Vert
\psi  \Vert  _{B^*}$
and choosing $u=0$ otherwise.

For the norm of $\wt B\,^*_2$
$$
\begin{aligned}
\Vert  (\varphi  ,\psi)  \Vert  _{\wt B\,^*_2} &= \sup\{\vert
\varphi  u+\psi  v\vert  :(\Vert  u\Vert  ^q_B + M\Vert  v\Vert
^q_B)^{1/q} =1\}\\
&\le \sup \{\Vert  \varphi  \Vert  _{B^*} \Vert  u\Vert  _B + \Vert
M^{-1/q}\psi \Vert  _{B^*}\Vert  M^{1/q}v\Vert  _B;\Vert  u\Vert
^q_B + M\Vert  v\Vert  ^q_B=1\}\\
&\le (\Vert  \varphi  \Vert  ^s_{B^*} + M^{-s/q}\Vert  \psi  \Vert
^s_{B^*})^{1/s} \\
&= (\Vert  \varphi  \Vert  ^s_{B^*} + m\Vert  \psi  \Vert
^s_{B^*})^{1/s}
\end{aligned}
$$
with $s=\frac{q}{q-1}$ and $m=M^{-1/(q-1)}.$  To show equality, we
choose $a\ge 0$ and $b\ge 0$ for which $a^q + Mb^q=1$ and
$(a^q,Mb^q)$ is proportional to $(\Vert  \varphi  \Vert
^s_{B^*},m\Vert  \psi  \Vert  ^s_{B^*}),$ and then choose $\Vert
u_n\Vert  _B=a,$ \, $\Vert  v_n\Vert  _B=b$ such that
$ \varphi  u_n\to a\Vert  \varphi  \Vert  _{B^*}$ and $\psi  v_n\to
b\Vert  \psi  \Vert  _{B^*}\,.$

The second assertion can be obtained in a similar way using the
operator $O$ on $(x,y)\in X\times X=\wt X$ given by
$$
O(x,y) = (x+y,x-y)
$$
and endowing $\wt X$ with the norms $\Vert  (x,y)\Vert  _{\wt X_1}
=\max(\Vert  x\Vert  _X,\Vert  y\Vert  _X)$ and $\Vert  (x,y)\Vert
_{\wt X_2} = (\Vert  x\Vert  ^s_X + m\Vert  y\Vert  ^s_X)^{1/s}$,
and hence~\eqref{Eq2.2} is satisfied by all $x,y\in B_1=X^*$ and the $B_1$ norm. In the terminology of~\cite[p.59]{Li-Tz} this implies that $B_1$ is uniformly smooth and hence (see~\cite[p.61, Prop.~1.e.2(ii)]{Li-Tz}) $X$ is uniformly convex. We note that~\eqref{Eq3.2} and the above now imply (see~\cite[p.61, Prop.~1.e.3]{Li-Tz}) that both $B$ and $X$ are reflexive. Therefore, $B_1=B$.
\end{proof}

As a corollary of Theorem \ref{Thm3.1} we show that the condition
(\ref{Eq1.2}) is satisfied by $L_p$ spaces.

\begin{corollary}\label{Cor3.2}
For $L_p$ with $1<p<\infty  $
\begin{equation}\label{Eq3.4}
\max\,(\Vert  F+G\Vert  _p,\Vert  F-G\Vert  _p)\ge (\Vert  F\Vert
^{\max(p,2)}_p + m\Vert  G\Vert  ^{\max(p,2)}_p)^{1/\max(p,2)}
\end{equation}
for some $m>0.$
\end{corollary}

\begin{proof}
We recall that for $L_p,$\;$1<p<\infty  ,$ (\ref{Eq3.2}) is valid
with $q=\min\,(p,2)$ (see \cite{Di88}) and use Theorem \ref{Thm3.1}.
\end{proof}

\begin{remark} \label{Rem3.3}
As (\ref{Eq3.1}) with $x=F$ and $y=G$ was shown to be equivalent to
(\ref{Eq3.2}) (see \cite{Di88}) and (\ref{Eq3.3}) was shown to be
dual to (\ref{Eq3.2}), the condition
\begin{equation}\label{Eq3.4n}
\begin{gathered}
\delta  _X(\varepsilon  )\ge K\varepsilon  ^s \q\text{where}\\
\delta  _X(\varepsilon  ) \equiv \inf(1-\Vert  \varphi  +\psi\Vert
_X/2 :\varphi  ,\psi\in X, \Vert  \varphi  \Vert  _X = \Vert
\psi\Vert  _X = 1, \Vert  \varphi  -\psi\Vert  _X =\varepsilon  )
\end{gathered}
\end{equation}
which is dual to (\ref{Eq3.1}) (see \cite[p.63]{Li-Tz}) is
equivalent to (\ref{Eq3.3}). Hence we note that the condition
(\ref{Eq3.3}) on (a given norm of) a Banach space $X$ means that $X$
has a modulus of convexity of at least power type $s$ (see
\cite[p.63]{Li-Tz}).
\end{remark}

\begin{remark}\label{Rem3.4}
For a space $B$ both the inequalities (\ref{Eq1.1}) and (\ref{Eq1.3})
depend on the norm and may not be valid for an equivalent norm.
However, the sharp Marchaud inequality or sharp converse inequality
is valid if it is valid for an equivalent norm.  It will be evident
that the validity of the sharp Jackson inequality and of the lower
estimate for the modulus of smoothness will, in the situations proved
in this paper for one norm of $B,$ imply their  validity for any
equivalent norm.
\end{remark}

\begin{remark}\label{Rem3.5}
On the face of it, it may seem that in Theorem~\ref{Thm3.1} we neglected to treat the situation when $q>2$. However, as~\eqref{Eq3.2} is equivalent to $\eta_B(\sigma)\le k\sigma^q$ (with $\eta_B(\sigma)$ of~\eqref{Eq3.1}), and as $\eta_B(\sigma)/\sigma^2$ is equivalent to a non-increasing function for any Banach space (see~\cite[p.64, Prop.~1.e.5]{Li-Tz}), a nontrivial Banach space (different from $\IR$ or $\{0\}$) for which~\eqref{Eq3.2} is satisfied with $q>2$ does not exist.
\end{remark}

We outline now the basic notations (and some facts) concerning
Orlicz spaces (see \cite{Ra-Re} and \cite[pp.265-280]{Be-Sh}) which
we will use in this section and later.  A Young function $\Phi$ is
an increasing convex function on $\IR_+$ satisfying $\Phi(0)=0.$
For a domain $\Omega  $ and a (positive) measure $d\mu  (x)$ the
Orlicz class $\SM(\Phi)$ and the Orlicz functional $M_\Phi(f)$ are
given by
\begin{equation}\label{Eq3.5}
\SM(\Phi) \equiv \Big\{ f:M_\Phi(f) \equiv \int_\Omega  \Phi(\vert
f(x)\vert  )d\mu  (x)<\infty  \Big\}.
\end{equation}

The Luxemburg norm of the Orlicz space is given by
\begin{equation}\label{Eq3.6}
\Vert  f\Vert  _{O_L(\Phi)} \equiv
\inf\Big(a>0:M_\Phi\big(\frac{\vert  f\vert  }{a}\big)\le 1\Big).
\end{equation}
$\Psi(y):\IR_+ \to \IR_+$ is the complementary Young function to
$\Phi(x)$ satisfying $\us{x\to\infty  }\lim\,\frac{\Phi(x)}{x} =
\infty  $  if
\begin{equation}\label{Eq3.7}
\Psi(y) =\sup \{xy - \Phi(x):x\ge 0\} , \q y\ge 0.
\end{equation}
The Orlicz norm of the Orlicz space is given by
\begin{equation}\label{Eq3.8}
\Vert  f\Vert  _{O(\Phi)} \equiv \sup\Big\{\int_\Omega  \vert
f(x)g(x)\vert  d\mu  (x): \int_\Omega  \Psi (\vert  g(x)\vert  )d\mu
 (x)\le 1\big\}.
\end{equation}
A Young function $\Phi$ satisfies the $\Delta  _2$ condition if for
some $K>0$
$$
\Phi(2x)\le K\Phi(x)\q\text{\rm for}\q x\ge x_0\ge 0 \q
(x_0=0\q\text{\rm when}\q \mu  (\Omega  )<\infty  ).
$$
A Young function $\Psi$ satisfies the $\nabla_2$ condition if for some $a>1$
$$
\Psi(x) \le \frac{1}{2a} \,\Psi(ax) \q\text{\rm for} \q x\ge x_0\ge
0 \q (x_0=0\q\text{\rm when}\q \mu  (\Omega  )<\infty  ).
$$

It is known that if $\Phi$ is a Young function, $\Psi$ given by
(\ref{Eq3.7}) is a Young function as well.  Also $\Vert  f\Vert
_{O_L(\Phi)} \le \Vert  f\Vert  _{O(\Phi)} \le 2\Vert  f\Vert
_{O_L(\Phi)}$
(see
\cite[Th.8.14, p.272]{Be-Sh}).  Moreover, if $\Phi$ satisfies the
$\Delta  _2$ condition, the complementary Young function $\Psi$
satisfies the $\nabla_2$ condition (see \cite[Cor.4, p.26]{Ra-Re}).

\begin{lemma}\label{Lem3.5}
Suppose $\Phi$ is a Young function and that $\Phi(u^{1/s})$ is
concave for some $1<s<\infty  .$  Then $\Psi(t^{1/q})$ is convex
where $\Psi$ is the complementary Young function and $\frac 1s
+\frac 1q =1.$
\end{lemma}

\begin{proof}
Let $g(u)\equiv \Phi(u^{1/s})$, $u\ge0$. Then
$$
g(u)=\inf_{z\ge0}(g_+'(z)(u-z)+g(z))
$$
and as $\lim_{u\to\infty}g(u)=+\infty$, for every $z\ge0$ we have $g_+'(z)>0$. In other words,
$g(u)=\inf_{(a,b)\in L}(au+b)$,
where $L$ is a subset of $(0,+\infty)\times\IR$.
Therefore, $\Phi(x)=\inf_{(a,b)\in L}(ax^s+b)$. By the definition of $\Psi$,
\[
\Psi(y)=\sup_{x\ge0}(xy-\Phi(x))
=\sup_{x\ge0}\sup_{(a,b)\in L}(xy-ax^s-b)=\sup_{(a,b)\in L}\sup_{x\ge0}(xy-ax^s-b).
\]
As $s>1$, the second supremum is achieved at $x=(\frac{y}{as})^{\frac1{s-1}}$. Hence,
\[
\Psi(y)=\sup_{(a,b)\in
L}\Bigl(\bigl(\frac1{(as)^{\frac1{s-1}}}-a(as)^{-q}\bigr)y^q-b\Bigr),
\]
which means that $\Psi(t^{1/q})$ is a supremum of a family of
functions  linear in $t,$
and therefore $\Psi(t^{1/q})$ is convex.
\end{proof}

\begin{lemma}\label{Lem3.6}
Suppose $\Phi(u^{1/s})$ is concave for some $s$, $2\le s<\infty$,
where $\Phi$ is a Young function satisfying the $\nabla_2$ condition.
Then there exist constants $A,m>0$ and a Young function
$\widetilde\Phi(u)$, such that $A^{-1}\Phi(u)\le\widetilde\Phi(u)\le
A\Phi(u)$, satisfying
\begin{equation}\label{star}
\max\{\Vert  {f+g}\Vert  _{O(\wt\Phi)},\Vert  {f-g}\Vert  _{O(\wt\Phi)}\}\ge
(\Vert  {f}\Vert  _{O(\wt\Phi)}^s+m\Vert  {g}\Vert  _{O(\wt\Phi)}^s)^{1/s},
\quad \text{for all } f,g\in \SM(\Phi).
\end{equation}
\end{lemma}

\begin{proof}
The complementary Young function $\Psi$ satisfies the
$\Delta_2$ condition, and $\Psi(t^{1/q})$ is convex for
$\frac1q+\frac1s=1$ by the previous lemma. Thus, we can apply
Lemma~2.2 of~[Di-Pr] for $B=O_L(\Psi)$ and $M=\Psi$, to find a Young function
$N=\wt\Psi$, equivalent to $\Psi$ such that
\begin{equation}
\label{star2}
\frac{\Vert{f+g}\Vert  _{O_L(\wt\Psi)}+\Vert{f-g}\Vert  _{O_L(\wt\Psi)}}{2}\le
(\Vert{f}\Vert  ^q_{O_L(\wt\Psi)} +L\Vert{g}\Vert  ^q_{O_L(\wt\Psi)})^{1/q},
\quad \text{for all } f,g\in \SM(\Psi)
\end{equation}
with $L>0$. Let $\wt\Phi$ be the complementary Young function of $\wt\Psi$.
The Young function $\wt\Phi$ is equivalent to $\Phi$ (\cite[Prop.2, p.15]{Ra-Re})
and the dual of $O_L(\wt\Psi)$ is isometric to $O(\wt\Phi)$
(\cite[Cor.9, p.111]{Ra-Re}). Hence, using Theorem \ref{Thm3.1}, \eqref{star2} implies
\eqref{star}.
\end{proof}

Now we will show examples of Young functions $\Phi$
for which there exists an equivalent Young function
$\wt\Phi$  such that $\wt\Phi(u^{1/s})$ is concave for some $s$,
$2\le s <\infty$, and which satisfies the $\nabla_2$ condition (consequently, the
corresponding Orlicz spaces will satisfy~(1.2)).

We intend to consider $\Phi(u)=u^r(1+|\ln u|)$ and $\Phi(u)=\max\{u^\alpha,u^\beta\}$
for appropriate values of $r, \alpha, \beta$. Note that these functions themselves
(being convex) cannot satisfy the condition that $g(u)\equiv \Phi(u^{1/s})$ is
concave for the following reason: $\Phi'_+(1)>\Phi'_-(1),$ and hence $g'_+(1)>g'_-(1)$.
However, with proper $s$, $g$ can be concave near $0$ and near $\infty$. Our task
is to ``patch'' these pieces together to construct an equivalent function $\wt\Phi$
satisfying the necessary conditions.

\begin{lemma}\label{Lem3.7}
Let $\Phi$ be a Young function such that
\begin{equation}
\label{co}
\Phi(u^{1/s}) \text{ is concave on }[0,a]\text{ and on }[b,\infty),
\end{equation}
where $0<a<b$, $s\ge2$. Then there is a Young function $\wt\Phi$
satisfying
\begin{equation}
\label{c1}
\wt\Phi(u)=c_1\Phi(u), \quad u\in[0,a],
\end{equation}
and
\begin{equation}
\label{c2}
\wt\Phi(u)=c_2+\Phi(u), \quad u\in[b,\infty),
\end{equation}
with some constants $c_1>0$ and $c_2,$ which is equivalent to
$\Phi(u)$ and
also $\wt\Phi(u^{1/s})$ is concave on $[0,\infty)$.
\end{lemma}

\begin{proof}
As $\Phi$ is convex, it is absolutely continuous and $\Phi'$ exists almost everywhere and is non-decreasing.
We choose $c_1$ to satisfy
\[
c_1\Phi'(a-)a^{\frac1s-1}=\Phi'(b+)b^{\frac1s-1}.
\]
We now define
\[
\phi(u):=\begin{cases}
c_1\Phi'(u), & u\in[0,a),\\
u^{1-\frac1s}\Phi'(b+)b^{\frac1s-1}, & u\in[a,b],\\
\Phi'(u), & u\in(b,\infty),
\end{cases}
\]
and $\wt\Phi(x):=\int_0^x\phi(u)\,du$. Clearly, \eqref{c1} and~\eqref{c2} are satisfied.
Also, as $\phi(a)=c_1\Phi'(a-)$, $\phi(b)=\Phi'(b+)$ and $\phi$ is increasing on $[a,b]$,
$\wt\Phi$ is a Young function. For $u\in[a,b]$ we obtain
\[
(\wt\Phi(u^{1/s}))'=\frac1s\phi(u)u^{\frac1s-1}=\frac1s\Phi'(b+)b^{\frac1s-1}=const=
(\wt\Phi(u^{1/s}))'|_{u=a-}=(\wt\Phi(u^{1/s}))'|_{u=b+}.
\]
Hence, $(\wt\Phi(u^{1/s}))'$ is non-increasing on $[0,\infty)$.

We observe that the resulting Young function $\wt\Phi$ is equivalent to $\Phi$.
\end{proof}

\begin{example}\label{Ex3.8}
Let $\Phi(u)=\max\{u^\alpha,u^\beta\}$, where $1<\alpha<\beta$.
Then $\Phi(u^{1/s})$ satisfies~\eqref{co} for any $s\ge\max\{2,\beta\}$.
\end{example}
\begin{proof}
We have
\[
\Phi(u^{1/s})=
\begin{cases}
u^{\alpha/s}, & u\le 1, \\
u^{\beta/s}, & u>1,
\end{cases}
\]
so both $\alpha/s$ and $\beta/s$ must not exceed $1$.
\end{proof}

\begin{example}\label{Eq3.9}
Let $\Phi(u)=u^r(1+|\ln u|)$, $r\ge(3+\sqrt{5})/2$ (which guarantees that $\Phi$ is a Young function).
Then $\Phi(u^{1/s})$ satisfies~\eqref{co} for any $s>r$ and does not satisfy~\eqref{co} with $s=r$.
\end{example}
\begin{proof}
We have
\[
\Phi'(u)=
\begin{cases}
u^{r-1}(r+1+r\ln u), & u>1, \\
u^{r-1}(r-1-r\ln u), & u<1,
\end{cases}
\]
and
\[
\Phi''(u)=
\begin{cases}
u^{r-2}(r^2+r-1+r(r-1)\ln u), & u>1,\\
u^{r-2}(r^2-3r+1-r(r-1)\ln u), & u<1.
\end{cases}
\]
Hence, $r\ge(3+\sqrt{5})/2$ implies convexity of $\Phi$. We further compute
\[
(\Phi(u^{\frac1s}))'=
\begin{cases}
\frac1s u^{\frac rs-1}(r+1+\frac rs \ln u), & u>1, \\
\frac1s u^{\frac rs-1}(r-1-\frac rs \ln u), & u<1,
\end{cases}
\]
and
\[
(\Phi(u^{\frac1s}))''=
\begin{cases}
\frac1s u^{\frac rs-2}(\frac rs (r+2)-r-1+\frac rs (\frac rs - 1) \ln u), & u>1, \\
\frac1s u^{\frac rs-2}(\frac rs (r-2)-r+1-\frac rs (\frac rs - 1) \ln u), & u<1.
\end{cases}
\]
Under the condition $s>r$, the function $(\Phi(u^{\frac1s}))''$ is clearly non-positive
for $u<1$ and also non-positive for $u>u_0$, where $u_0$ is such that
$\frac rs (\frac rs - 1) \ln u_0 = -1$. If $r=s$, then $(\Phi(u^{\frac1s}))''=1$ for $u>1$.
\end{proof}

\begin{example}\label{Ex3.10}
The Zygmund spaces $L_p(\text{\rm Log}\, L)^\alpha  .$
Let $\Phi(u)=u^p(\ln(2+u))^{\alpha p}$, $\alpha p\ge1$, $p\ge 1$
(see \cite[Def.6.11, p.252]{Be-Sh}).
Then $\Phi(u^{1/s})$ satisfies~\eqref{co} for any $s>p$ and does not satisfy~\eqref{co} with $s=p$.
\end{example}
\begin{proof}
We find
\[
\Phi'(u)=pu^{p-1}\ln^{\alpha p-1}(2+u)\Bigl(\ln(2+u)+\frac{\alpha u}{2+u}\Bigr),
\]
and hence,
\[
(\Phi(u^{\frac1s}))'=\frac ps u^{\frac ps-1} \ln^{\alpha p -1}(2+u^{\frac 1s})
\Bigl(\ln(2+u^{\frac 1s})+\frac{\alpha u^{\frac 1s}}{2+u^{\frac 1s}}\Bigr).
\]
Differentiating once more, we obtain
\begin{equation}
\label{eqt}
(\Phi(u^{\frac1s}))''=\frac ps u^{\frac ps -2} \ln^{\alpha p -1}(2+u^{\frac 1s})
 \Bigl(\ln(2+u^{\frac 1s})+\frac{\alpha u^{\frac 1s}}{2+u^{\frac 1s}}\Bigr)
 \Bigl(\frac ps -1 + D(u)\Bigr),
\end{equation}
where
\[
D(u)=\frac{u^{\frac1s}}{s(2+u^{\frac1s})}\Bigl(\frac{\alpha p-1}{\ln(2+u^{\frac1s})}
+\frac{1+\frac{2\alpha}{2+u^{\frac 1s}}}{\ln(2+u^{\frac 1s})+\frac{\alpha u^{\frac 1s}}{2+u^{\frac 1s}}}\Bigr).
\]
The sign of $(\Phi(u^{\frac1s}))''$ for $u>0$ is determined by the sign of
the last factor of the right hand side of ~\eqref{eqt}. As $D(u)$ is positive for $u>0$, if $s=p$,
then $(\Phi(u^{\frac1s}))''>0$. If $s>p$, then $\frac ps-1<0$ and we can
choose $a,b$ to satisfy~\eqref{co} since
\[
\lim_{u\to0+}D(u)=0 \quad\text{and}\quad
\lim_{u\to\infty}D(u)=0.
\]
\end{proof}

Note that in all the above examples it is easy to verify that $\Phi$ satisfies
the $\Delta_2$ and the $\nabla_2$ conditions.

\section{Applications using Holomorphic semigroups}\label{Sec4}

The operators $T(t),\; T(t):B\to B$ for $t\in [0,\infty  ) =
\IR_+$form a $C_0$ semigroup if

\noi
$T(t+s)f = T(t)T(s)f$ and $\us{t\to
0+}\lim\,\Vert  T(t)f-f\Vert  _B=0.$ \,$\{T(t)\}_{t\ge 0}$ is a
semigroup of contractions if $\Vert  T(t)f\Vert  _B\le \Vert  f\Vert
 _B$ for all $t\in \IR_+$ and $f\in B.$  The infinitesimal generator
$\CA$ related to the semigroup $\{T(t)\}_{t\ge 0}$ is given by
\begin{equation}\label{Eq4.1}
\CA f \equiv \us{t\to 0+}{\text{\rm s\,-}\lim}\; \frac{T(t)f-f}{t}\,,
\end{equation}
(where $\us{t\to 0+}{\text{\rm s\,-}\lim}\,g_t=\varphi  $ if $\Vert  g_t-\varphi
\Vert  _B\to 0$ as $t\to 0+)$
and the domain of $\CA,\;\SD(\CA),$ consists of all $f$ such that the
limit in (\ref{Eq4.1}) exists.  A holomorphic semigroup is a
semigroup satisfying
\begin{equation}\label{Eq4.2}
T(t)f\in \SD(\CA) \q\text{\rm for }\q t>0\q\text{\rm and}\q t\Vert
\CA T(t)f\Vert  _B\le N\Vert  f\Vert  _B
\end{equation}
with $N$ independent of $t$ and $f.$  (Note that (\ref{Eq4.2}) is
essentially a Bernstein-type inequality.)

It was proved (see \cite[Th.5.1, p.74]{Di-Iv}) that for a
holomorphic $C_0$ semigroup of contractions we have
\begin{equation}\label{Eq4.3}
\big\Vert  \big(T(t) - I\big)^r f\big\Vert  _B\approx \us{g\in
\SD(\CA^r)}\inf\, (\Vert  f-g\Vert  _B + t^r \Vert  \CA^rg\Vert
_B)\equiv K_{\CA^r} (f,t^r)_B,
\end{equation}
which is a strong converse inequality of type $A$ in the terminology
of \cite{Di-Iv}.  We recall that by $A(t)\approx E(t)$ one means
$C^{-1}A(t)\le E(t)\le CA(t).$
Using (\ref{Eq4.3}) and general properties of
$K\text{\rm-functionals,}$ we have for holomorphic semigroups
\begin{equation}\label{Eq4.4}
\big\Vert  \big(T(t)-I\big)^r f\big\Vert  _B\approx \us{0<u\le
t}\sup\,\big\Vert  \big(T(u)-I\big)^rf\big\Vert  _B.
\end{equation}

As a corollary of Theorem \ref{Thm2.1} and (\ref{Eq4.3}) we obtain
the following result.

\begin{theorem}\label{Thm4.1}
Suppose that $\{T(t)\}_{t\ge 0}$ is a holomorphic $C_0$ semigroup of
contractions on a Banach space $B$ and that $B$ satisfies the condition
{\rm (\ref{Eq1.2})} for some $2\le s< \infty  $ and $m>0.$  Then for any
integer $r$
\begin{equation}\label{Eq4.5}
K_{\CA^r}(f,t^r)_B\ge C\Big\{\sum^\infty_{j=1} 2^{-jrs} K_{\CA^{r+1}}
(f,2^{j(r+1)}t^{r+1})^s_B\Big\}^{1/s}.
\end{equation}
\end{theorem}

\begin{proof}
We use (\ref{Eq2.1}) with $T=T(t)$ and $T^{2^\ell}= T(t2^\ell),$ to
which we apply (\ref{Eq4.3}) (for both $r$ and $r+1),$ which yields
$\big \Vert  \big(T(t)-I\big)^rf\big\Vert  _B\le
C_1K_{\CA^r}(f,t^r)_B$ and $\big\Vert
\big(T(t2^j)-I\big)^{r+1}f\big\Vert  _B\ge
C_2K_{\CA^r}(f,2^{j(r+1)}t^{r+1})_B$ to complete the proof of
(\ref{Eq4.5}).
\end{proof}

The usefulness of Theorem \ref{Thm4.1} is clearly demonstrated by
applying it to the Gauss-Weierstrass semigroup of operators (see for
instance \cite[p.261]{Bu-Be}) given by
\begin{equation}\label{Eq4.6}
W(t)f(x)\equiv \frac{1}{(4\pi t)^{d/2}} \;\int_{\IR^d
}\,\exp\Big(\frac{-\vert  x-u\vert  ^2}{4t}\Big)f(u)du.
\end{equation}

\begin{theorem}\label{Thm4.2}
Suppose that $B,$ a Banach space of functions on $\IR^d,$ satisfies
{\rm (\ref{Eq1.2}) and (\ref{Eq1.3})} and that $B\subset \CS^\pr$
which means that
$B$ is
continuously imbedded in the Schwartz space of tempered
distribution. Then
\begin{equation}\label{Eq4.7}
\begin{aligned}
\big\Vert  \big(W(t)-I\big)^rf\big\Vert  _B &\approx \us{u\le
t}\sup\, \big\Vert  \big(W(u)-I\big)^r f\big\Vert  _B\approx
\us{g\in \SD(\Delta  ^r)}\inf\, (\Vert  f-g\Vert  _B+t^r\Vert
\Delta  ^r g\Vert  _B)\\
&\equiv K_{\Delta  ^r}(f,t^r)_B,
\end{aligned}
\end{equation}

\begin{equation}\label{Eq4.8}
K_{\Delta  ^r} (f,t^r)_B\ge C\Big\{\sum^\infty_{j=1} 2^{-jrs}
K_{\Delta  ^{r+1}} (f,2^{j(r+1)}t^{r+1})^s_B\Big\}^{1/s}
\end{equation}
and
\begin{equation}\label{Eq4.9}
K_{\Delta  ^r} (f,t^r)_B \ge C\Big\{\sum^\infty_{j=1} 2^{-jrs}
E_{2^{j/2}t^{1/2}}(f)^s_B\Big\}^{1/s}
\end{equation}
where $\Delta  $ is the Laplacian and $E_\lambda  (f)_B$ is given by
\begin{equation}\label{Eq4.10}
E_\lambda  (f)_B = \inf\,\{\Vert  f-\varphi  _\sigma  \Vert
_B:\,\supp\; \wh\varphi  _\sigma  (y)\subset (y:\vert  y\vert  \le
\lambda  )\},
\end{equation}
where $\wh \varphi  _\sigma  $ is the Fourier transform of $\varphi
_\sigma  \,.$
\end{theorem}

\begin{proof}
For $f\in B$ satisfying (\ref{Eq1.3}) we may use the Riemann vector
valued integration in (\ref{Eq4.6}) to obtain for all $f\in B$
\begin{equation}\label{Eq4.11}
\begin{aligned}
\Vert  W(t)f\Vert  _B &\le \Vert  f\Vert  _B, \q \us{t\to 0+}\lim\,
\Vert  W(t)f-f\Vert  _B =0, \\
\Delta  ^rW(t)f&\in B\q\text{\rm and}
\q \Vert  \Delta  ^r W(t)f\Vert  _B\le \Big(\frac dt\Big)^r\Vert
f\Vert  _B.
\end{aligned}
\end{equation}
For $\varphi  \in \CS,$ the Schwartz space of test functions,
straightforward computation implies $\frac{W(t)\varphi  -\varphi
}{t} - \Delta  \varphi   \to 0$ in $\CS$ and hence in $B^*,$ the
dual to $B.$  Therefore, whenever $f\in \SD(\Delta  ),$ that is when
$\Delta  f$ exists in the $\CS^\pr$ sense and $\Delta  f\in B,$ we have
$$
\Big\la \,\frac{W(t)f-f}{t} - \Delta  f,\varphi \, \Big\ra = \Big\la
\,f,\,\frac{W(t)\varphi  -\varphi  }{t} - \Delta  \varphi  \,\Big\ra
\to 0
$$
for all $\varphi  \in \CS.$  For $f\in \SD(\CA)$ with $\CA$ the
infinitesimal generator of $W(t)$

\noi
$\big\Vert  \,\frac{W(t)f-f}{t} -
\CA f\big\Vert  _B\to 0$ and $\big\Vert  \,\frac{W(t+s)f-W(s)f}{t} -
\CA W(s)f\big\Vert  _B\to 0.$  As $W(s) f\in \SD (\Delta  )$ and
$W(s) f \in \SD(\CA),$ we have
$ \us{t\to 0+}\lim\,\Big\la\,\frac{W(t+s)f-W(s)f}{t}\, - \Delta
W(s)f,\varphi  \,\Big\ra = 0$  and
$\us{t\to 0+}\lim\,\Big\la\,\frac{W(t+s)f-W(s)}{t}\, - \CA
W(s)f,\varphi  \,\Big\ra =0$ for all $\varphi  \in \CS \cap B^*,$
and hence $\Delta  W(s)f = \CA W(s)f$ for all $s>0.$

For $f\in \SD (\Delta  )$ we can now write
$$
\Vert  W(t+s)f -W(s)f\Vert  _B \le t\Vert  \CA W(s)f\Vert  _B =
t\Vert  \Delta  W(s)f\Vert  _B\le t\Vert  \Delta  f\Vert
_B\q\text{\rm for all}\q s>0,
$$
and using (\ref{Eq4.11}), $\Vert  W(t)f -f \Vert  _B\le t\Vert
\Delta  f\Vert  _B.$

Similarly, for $g\in \SD (\Delta  ^r)$
$$
\big\Vert  \big(W(t)-I\big)^r g\big\Vert  _B\le t^r\Vert  \Delta  ^r
g\Vert  _B.
$$
The above directly implies the inequality $\big\Vert
\big(W(t)-I\big)^r f\big\Vert  _B\le CK_{\Delta  ^r}(f,t^r)_B.$  The
proof of the inequality $K_{\Delta  ^r} (f,t^r)_B\le C\big\Vert
\big(W(t)-I\big)^rf\big\Vert  _B$ follows exactly the proof of
Theorem 5.1 in \cite{Di-Iv}, replacing $\CA$ by $\Delta  ,$ which as
it operates on $g=-\os r{\us{k=1}\sum} (-1)^k
\begin{binom}{r}{k}\end{binom} T(kmt)f,$ is the same.

For $B=L_p(\IR^d),$ \; $1\le p<\infty  ,$\; $\Delta  =\CA$ (see the
proof in \cite[Th.4.311, p.261]{Bu-Be}).

The inequality (\ref{Eq4.8}) now follows from Theorem \ref{Thm4.1} as
(\ref{Eq4.7}) implies $K_{\CA^r}(f,t^r)_B\approx K_{\Delta
^r}(f,t^r)_B$ for all $f,r$ and $t.$  The inequality (\ref{Eq4.9})
follows from (\ref{Eq4.8}) and the inequality
\begin{equation}\label{Eq4.12}
E_\lambda  (f)_B\le CK_{\Delta  ^r}(f,\lambda  ^{-2r})_B.
\end{equation}
The inequality (\ref{Eq4.12})
 was proved in \cite[(2.9), p.271]{Da-Di04} for $B=L_p(\IR^d),$
\; $1\le p\le \infty  .$  In fact, (\ref{Eq4.12}) follows for any
$B$ satisfying (\ref{Eq1.3}), as all we need in the proof of
\cite[271-272]{Da-Di04} is that the linear convolution operators
$R_{\lambda  ,\ell,b}f$ there satisfy
\begin{equation}\label{Eq4.13}
\Vert  R_{\lambda  ,\ell,b}f\Vert  _B\le C_1\Vert  f\Vert
_B\q\text{\rm and}\q \Vert  R_{\lambda  ,\ell,b} g-g\Vert  _B \le
C_2\,\frac{1}{\lambda  ^{2\ell}} \, \Vert  \Delta  ^\ell g\Vert
_B\,.
\end{equation}

We define $F=f*\varphi  $ with $\varphi  \in B^*$ and $\Vert
\varphi  \Vert  _{B^*} =1$ such that $\varphi  $ satisfies
$$
\begin{aligned}
\Vert  R_{\lambda  ,\ell,b}f\Vert  _B -\varepsilon  &\le \vert
R_{\lambda  ,\ell,b}F(0)\vert  \le \us x\sup\, \vert  R_{\lambda
,\ell,b}F(x)\vert  \\
&= \Vert  R_{\lambda  ,\ell,b}F\Vert  _\infty   \le C_1\Vert  F\Vert
 _\infty  \le C_1 \Vert  f\Vert  _B\,.
\end{aligned}
$$

Similarly, we obtain the second inequality of (\ref{Eq4.13}) using
$G=g*\varphi  .$
\end{proof}

\begin{remark}\label{Rem4.3}
For $L_p(\IR^d)$ a somewhat more general result than in Theorem
\ref{Thm4.2} was proved in \cite[Theorem 7.1]{Da-Di-Ti} using a
completely different method.  Here the proof is much simpler and
applies to a wide class of Orlicz spaces (see Section \ref{Sec3}),
and perhaps to other spaces that satisfy (\ref{Eq1.2}) with some
norm of $B$ that satisfies (\ref{Eq1.3}) at the same time.  Orlicz
spaces described in Section \ref{Sec3} satisfy (\ref{Eq1.2}) with
the same norm for which (\ref{Eq1.3}) is valid.
\end{remark}

\begin{remark}\label{Rem4.4}
Using the monotonicity of $K_{\Delta  ^{r+1}}(f,u)_B$ and of
$E_u(f)_B\,,$ one can obtain the following equivalent form of
(\ref{Eq4.8}) and (\ref{Eq4.9}), which may appear more traditional:
$$
K_{\Delta  ^r}(f,t^r)_B\ge Ct^r\Big\{\int^\infty_t u^{-rs} K_{\Delta
^{r+1}}(f,u^{r+1})^s_B\;\frac{du}{u}\Big\}^{1/s} \leqno (4.8)^\pr
$$
and
$$
K_{\Delta  ^r}(f,t^r)_B\ge Ct^r\Big\{\int^\infty_{t^{1/2}} u^{-2rs}
E_{u^{-1}}(f)^s_B\;\frac{du}{u}\,\Big\}^{1/s}.
\leqno(4.9)^\pr
$$
For $B= L_p(\IR^d),$ \; $1<p<\infty  ,$\, $\omega  ^r(f,t)_p$ and
$K_{\Delta  ^r}(f,t^r)_p$ given by (\ref{Eq1.5}) and (\ref{Eq4.7})
respectively satisfy $\omega  ^{2r}(f,t)_p\approx K_{\Delta
^r}(f,t^{2r})_p$ and hence (\ref{Eq4.8}) can take the form
$$
\omega  ^{2r}(f,t)_p\ge Ct^{2r}\Big\{\int^\infty_t u^{-2rs} \omega
^{2r+2}(f,u)^s_p\;\frac{du}{u}\,\Big\}^{1/s}, \q s=\max\,(p,2).
\leqno(4.8)^{\pr\pr}
$$

In fact, the result of Theorem \ref{Thm4.2} is given as an example
of use of Theorem \ref{Thm4.1}, and the same method can be used for
many semigroups that are given by positive convolution operators on
$\IR^d$ or $\IT^d,$\, $d=1,2,\dots\,.$

In the next section we will give applications relating to
holomorphic semigroups generated by multipliers.
\end{remark}

\section{Ces\`aro summability and holomorphic semigroups}\label{Sec5}
For the purpose of this section, $H_k$ are eigenspaces of a
self-adjoint operator $P(D),$ and $\lambda  _k$ the eigenvalues of
$P(D),$ satisfy $0\le \lambda  _k,$\; $\lambda  _k < \lambda
_{k+1}\,.$  Furthermore, for our space $B$ we assume that
$H_k\subset B,$ \, $H_k\subset B^*$ and that $\spa\;(\cup H_k)$ is
dense in $B.$  The expansion of $f$ is given by
\begin{equation}\label{Eq5.1}
f\sim \sum^\infty  _{k=0}\, P_k f
\end{equation}
where $P_k f$ is the projection of $f$ on $H_k$ in the $L_2$ sense
(see \cite[(2.2)]{Di98}).  It was shown in \cite{Da-Di05} that if
the Ces\`aro summability of some order $\ell$ is a contraction in
$B,$ that is
\begin{equation}\label{Eq5.2}
\Vert  C^\ell_n f\Vert  _B\le \Vert  f\Vert  _B
\end{equation}
for
\[
C_n^\ell f\equiv\frac1{A_n^\ell}\sum_{k=0}^\infty A_{n-k}^\ell P_kf
\quad\text{where}\quad A_m^\ell\equiv\frac{(m+\ell)!}{\ell!m!},
\]
then $T(t)f$ given by
\begin{equation}\label{Eq5.3}
T(t)f = \sum^\infty  _{k=0} e^{-kt} P_kf \q\text{\rm for} \q t>0
\q\text{\rm and} \q T(0)f=f
\end{equation}
is a holomorphic $C_0$ semigroup of contractions with its
infinitesimal generator given by
\begin{equation}\label{Eq5.4}
{\cal A} f\sim \sum^\infty  _{k=1} -kP_k f, \q \SD(\CA) =\{f\in B:
\exists\, g\in B\;\text{\rm such that}\;
P_k g =
-k P_k f\q\text{\rm for all}\q k\}.
\end{equation}

The following theorem will establish among other facts that the
positivity of $C^\ell_n f$ implies that it is a contraction in
Orlicz spaces with the Luxemburg norm as well as with the Orlicz norm.  We
remind the reader that if an operator is a contraction on a space with respect to
a given norm, it does not imply that it is a contraction with an
equivalent norm.

\begin{theorem}\label{Thm5.1}
Suppose $Of(x)$ is given by
\begin{equation}\label{Thm5.5}
Of(x) =\int_\Omega   f(y)G(x,y)w(y)dy
\end{equation}
where $G(x,y) = G(y,x) \ge 0,$\; $w(y)\ge 0$ and $\int_\Omega
G(x,y)w(y)dy = 1.$  Then $Of$ is a contraction with respect to the
Luxemburg norm given by
\begin{equation}\label{Eq5.6}
\Vert  f\Vert  _{O_L(\Phi)} = \inf\,\Big\{a\in \IR_+\,:\int_\Omega
\Phi\Big(\frac{\vert  f(x)\vert  }{a}\Big) w(x)dx\le 1\Big\}
\end{equation}
and with respect to the Orlicz norm given by
\begin{equation}\label{Eq5.7}
\Vert  f\Vert  _{O(\Phi)} =\sup\,\Big\{\int_\Omega  \vert  f(x)g(x)\vert
w(x)dx:\,\int_\Omega  \Psi (\vert  g(x)\vert  ) w(x)dx\le 1\Big\}
\end{equation}
where $\Phi$ and $\Psi$ are associate Young functions.
\end{theorem}

\begin{proof}
For $a\in \IR_+\,,$ which is close to the infimum in (\ref{Eq5.6}), we
write
$$
\int_\Omega  \Phi\Big(\Big\vert\,\frac 1a\, Of(x)\Big\vert
\Big)w(x)dx \le
\int_\Omega  \Phi\Big(\,\frac 1a\int_\Omega  \vert  f(y)\vert
G(x,y)w(y)dy\Big)w(x)dx \equiv I.
$$
Using Jensen's inequality, the convexity of $\Phi$ and $\int
G(x,y)w(y)dy =1,$ we have
$$
\begin{aligned}
I &\le \int_\Omega  \int_\Omega  \Phi\Big(\,\frac{\vert  f(y)\vert
}{a}\,\Big) G(x,y)w(y)dy\,w(x)dx\\
&= \int_\Omega  \,\Phi\Big(\,\frac{\vert  f(y)\vert  }{a}\,\Big)
w(y)dy \int_\Omega  G(x,y)w(x)dx\\
&= \int_\Omega  \Phi\Big(\,\frac{\vert  f(y)\vert  }{a}\,\Big)w(y)dy,
\end{aligned}
$$
which completes the proof for the Luxemburg norm of the Orlicz
space. We now write
$$
\begin{aligned}
\int_\Omega  \vert  Of(x)\vert  \,\vert  g(x)\vert  w(x)dx
&\le \int_\Omega  \vert  g(x)\vert  \int_\Omega  G(x,y) \vert
f(y)\vert  w(y)dy\,w(x)dx\\
&= \int_\Omega  \vert  f(y)\vert w(y)  \,\int_\Omega\vert  g(x)\vert  G(x,y)w(x)dx\,dy.
\end{aligned}
$$
As $\Psi  $ is also a Young function and is convex, we have
$$
\int_\Omega  \Psi \Big\{\int_\Omega  \vert  g(x)\vert
G(x,y)w(x)dx\Big\}w(y)dy \le \int_\Omega  \Psi \big(\vert  g(y)\vert
 \big) w(y)dy \le 1,
$$
and hence our result follows.
\end{proof}

For $L_p(\Omega  )$ with weight $w(x)\ge 0$ the proof is easier as
it follows directly from H\"older's inequality, but the result for
$L_p$ is included in the more intricate proof of Theorem
\ref{Thm5.1}.

Clearly, the positivity of the Ces\`aro summability in the above
context implies that
$$
C^\ell _n f(x) =\int_\Omega  f(y)G_{n,\ell}(x,y)w(y)dy
$$
where $G_{n,\ell}(x,y) = G_{n,\ell}(y,x),$\; $G_{n,\ell}(x,y)\ge 0,$
\; $w(y)\ge 0,$ and when $1\in H_0\,,$ also $\int G_{n,\ell}(x,y)w(y)dy
= 1.$

\begin{theorem}\label{Thm5.2}
Suppose $H_k,$ $\lambda  _k$ and $P_kf$ are as described at the
beginning of this section, $B$ is an Orlicz space which satisfies
{\rm (\ref{Eq1.2})} (for some $s,\; 2\le s<\infty  )$ with a
Luxemburg norm or Orlicz norm, $C^\ell_n$ is positive for some
$\ell,$  $1\in H_0$ and $\lambda  _k$ is a polynomial in $k$ of
degree $b.$  Then
\begin{equation}\label{Eq5.8}
K_{\CA^r}(f,t^r)_B\approx \,\inf\,\big\{\Vert  f-g\Vert  _B + t^r \Vert
P(D)^{r/b}g\Vert  _B: \,g\in \SD\big(P(D)^{r/b}\big)\big\},
\end{equation}
\begin{equation}\label{Eq5.9}
K_{\CA^r}(f,t^r)_B\ge C\Big\{\sum^\infty_{j=1}
2^{-jrs}K_{\CA^{r+1}}(f,t^{r+1}2^{j(r+1)})^s_B\Big\}^{1/s}
\end{equation}
and
\begin{equation}\label{Eq5.10}
K_{\CA^r}(f,2^{-nr})_B \ge C\Big\{\sum^n_{j=1} 2^{-jrs}
E_{2^{n-j}}(f)_B^s\Big\}^{1/s}
\end{equation}
where
$$
E_n(f) = \,\inf\,\Big\{\Vert  f-\varphi  \Vert  _B:\,\varphi  \in
\;\spa\;\bigcup^n_{k=0} \,H_k\Big\}.
$$
\end{theorem}

\begin{proof}
The proof of (\ref{Eq5.8}) follows the proof in \cite[Th.4.3, p.83]{Da-Di07} where the result is proved for $L_p$ spaces. In fact,
the same proof works for Banach spaces $B$ for which some order of
the Ces\`aro summability is bounded, which implies the realization
result (see \cite[Th.6.2 and Th.7.1]{Di98}, and that result is the
key ingredient for the proof in \cite[Th.4.3]{Da-Di07}.  We now show
\begin{equation}\label{Eq5.11}
E_n(f)_B \le C\,\inf\,\big(\Vert  f-g\Vert  _B +
n^{-r-1} \Vert  P(D)^{(r+1)/b} g\Vert  _B\big)
\approx K_{\CA^{r+1}}\big(f,n^{-r-1}\big)_B.
\end{equation}
The first inequality of (\ref{Eq5.11}) follows from
\cite[Th.4.1]{Di98} when we recall that $\lambda  _k\ge 0, $ and $\lambda
_k\approx k^b,$ (essentially $P(D),$ \,$(r+1)/b$ and $n^b$ are
$-P(D),$ $\alpha  m$ and $\lambda  $ respectively in \cite{Di98}).
The second equivalence is treated in detail in \cite[Section 4]{Da-Di07}. Using (\ref{Eq5.11}), we may deduce (\ref{Eq5.10}) from
(\ref{Eq5.9}), which in turn is a direct application of
(\ref{Eq4.5}).
\end{proof}

\begin{remark}\label{Rem5.2}
Similar to what we stated in Remark \ref{Rem4.4}, one can give a
different form of (\ref{Eq5.9}) and (\ref{Eq5.10}).  For example, we
have
$$
K_{\CA^r}(f,t^r)_B\ge Ct^r\Big\{\sum^{[1/t]}_{n=1} \,
n^{-rs-1}E_n(f)^s_B\Big\}^{1/s}. \leqno(5.10)^{\pr}
$$
\end{remark}

\section{Sharp Jackson theorem for polynomials on a simplex}
\label{Sec6}
For the simplex $S\in \IR^d$
\begin{equation}\label{Eq6.1}
S = \Big\{\bx = (x_1,\dots,x_d):\, x_i \ge 0\q x_0=1-\sum^d_{i=1} x_i
  \ge 0\Big\},
\end{equation}
the Jacobi weight is given by
\begin{equation}\label{eq6.2}
W_{\balpha  }(\bx) = x^{\alpha  _0}_0\dots x^{\alpha  _d}_d, \q
\balpha
=
(\alpha  _0,\dots,\alpha  _d), \q \alpha  _i >\,-\,\frac 12\,.
\end{equation}
The self-adjoint differential operator (see \cite[p.226]{Di95}) on
$S$ with weight $w_{\balpha}(\bx)$ is given  by
\begin{equation}\label{Eq6.3}
P_{\balpha} (D) = -\sum_{\pmb{\xi}  \in E_S} w_{\balpha} (\bx)^{-1}
\;\pd{}{\pmb{\xi} }\,\wt d(\pmb{\xi}  ,\bx)w_\balpha
(\bx)\,\pd{}{\pmb{\xi}  } \equiv
\sum_{\pmb{\xi}  \in E_S} P_{\balpha  ,\pmb{\xi}  }(D)
\end{equation}
where $E_S$ is the set of directions parallel to the edges of $S,$
and $\wt d(\pmb{\xi}  ,\bx)$ is given by
\begin{equation}\label{Eq6.4}
\wt d(\pmb{\xi}  ,\bx) =\us{\os{\lambda  \ge 0}{\us{\bx +\lambda
\pmb{\xi} \in S}{}}}\sup\, d(\bx,\bx
+\lambda  \pmb{\xi}  ) \us{\stackrel{\lambda  \ge 0}{\bx -\lambda
\pmb{\xi}
\in S}}\sup\, d(\bx,\bx-\lambda  \pmb{\xi} )
\end{equation}
using the Euclidean distance $d(\bx,\by).$

For  $\Pi_k$ the polynomials of total degree $\le k$ we have $\Pi_k
= H_0 \oplus \dots \oplus H_k$ where
\begin{equation}\label{Eq6.5}
P_{\balpha }(D)\varphi  =\ell\Big(\ell+ d +\sum^d_{i=0}
\alpha  _i\Big)\varphi  \equiv \lambda  _\ell\varphi  , \q \varphi  \in H_\ell.
\end{equation}
Defining the $K\text{\rm -functional}$ on $S$ by
\begin{equation}\label{Eq6.7}
K_r \big(f,P_{\balpha}(D)^{r/2},t^r\big)_B = \inf\Big(\Vert  f-g\Vert
_B +t^r\Vert  P_\balpha (D)^{r/2}g\Vert  _B\,:\,g\in \SD\big(P_\balpha
(D)^{r/2}\big)\Big)
\end{equation}
where for $\beta  \in [0,\infty  )$\; $P_{\balpha}(D)^\beta  $ is
given for $\beta  >0$ by

\begin{equation} \label{Eq6.8}
P_{\balpha}(D)^\beta  f =\sum^\infty  _{\ell=1} \lambda  ^\beta
_\ell P_\ell f, \q f\in \SD\big(P_{\balpha}(D)^\beta  \big) \q
\text{\rm if} \q \exists\, \psi\in B\q P_\ell \psi = \lambda  ^\beta
_\ell P_\ell f
\end{equation}
where $P_\ell \varphi  $ is the $L_2$ projection of $\varphi  $ onto
$H_\ell\,.$
We can now deduce the sharp Jackson inequality for polynomials and
lower estimate for $K\text{\rm -functionals}$ on the simplex.

\begin{theorem}\label{Thm6.1}
Suppose $B$ is a weighted $L_p$ or an Orlicz space on the simplex
$S$ satisfying {\rm (\ref{Eq1.2})} for some $2\le s<\infty  .$  Then
\begin{equation}\label{Eq6.9}
K_r\big(f,P_{\bal}(D)^{r/2},t^r\big)_B \ge C\Big\{\sum^\infty_{j=1}
2^{-jrs}K_{r+1}\Big(f,P_{\bal}(D)^{(r+1)/2},
t^{r+1}2^{j(r+1)}\Big)^s_B  \Big\}^{1/s}
\end{equation}
and
\begin{equation}\label{Eq6.10}
K_r\big(f,P_{\bal}(D)^{r/2},2^{-nr}\big)_B\ge C_1
2^{-nr}\Big\{\sum^n_{j=0} 2^{j rs}E_{2^j}(f)^s_B\Big\}^{1/s}
\end{equation}
where $S,\,P(D),\, K_r\big(f,P_{\bal}(D)^{r/2},t\big)_B$ and
$P_{\bal}(D)^{r/2}$ are given by {\rm (\ref{Eq6.1}), (\ref{Eq6.3}),
(\ref{Eq6.7})} and {\rm (\ref{Eq6.8})} respectively and $E_n (f)_B$ is given by
\begin{equation}\label{Eq6.11}
E_n (f)_B =\;\inf\,(\Vert  f-P\Vert  _B:\, P\in \Pi_n).
\end{equation}
\end{theorem}

\begin{proof}
We follow \cite[Cor.7.4.2, p.273]{Du-Xu}, which implies the
positivity of the Ces\`aro summability $C^\delta  _n\,,$ provided that
$\delta  $ is large enough.  The use of Theorem \ref{Thm5.2} will
complete the proof of (\ref{Eq6.9}), when we recall that $\lambda
_\ell = \ell \big(\ell + d +\os d{\us{i=0}\sum}\, \alpha  _i\big)$
is a polynomial of degree $b=2$ in $\ell.$  The proof of
(\ref{Eq6.10}) follows from the boundedness of the Ces\`aro
summability which implies (see \cite[Th.6.1]{Di98})
\begin{equation}\label{Eq6.12}
E_n(f) _B\le CK_{r+1}\big(f,P_{\bal}(D)^{(r+1)/2}, n^{-r-1}\big)_B
\end{equation}
and hence (\ref{Eq6.10}) can be deduced from (\ref{Eq6.9}).
\end{proof}

For $d=1$ and $B=L_p$ with Jacobi weights, Theorem \ref{Thm6.1} was
proved in \cite[Th.6.1]{Da-Di-Ti}.

\section{Sharp Jackson inequality on the sphere}\label{Sec7a}

The result of this section was proved for $L_p(S^{d-1}),$ \,
$1<p<\infty  $ in \cite[Th.8.1]{Da-Di-Ti}.  Here we will give
an alternative proof which yields an extension to a class of Banach
spaces that include many Orlicz spaces.

The Laplace-Beltrami
operator $\wt\Delta  $ on the unit sphere $S^{d-1} =\{\bx \in
\IR^d:\vert  \bx\vert  ^2 \equiv x^2_1 +\dots + x^2_d=1\}$ is given by
\begin{equation}\label{Eq7a.1}
\wt \Delta  f(\bx) = \Delta F(\bx), \q \bx\in S^{d-1} \q
\text{where } F(\bx)= f\Big(\frac{\bx}{\vert  \bx\vert
}\Big) \text{ and } \Delta  =
\pd{^2}{x^2_1} + \dots +\pd{^2}{x^2_d}\,.
\end{equation}
The eigenspace $H_k$ of spherical harmonic polynomials of degree $k$
on $S^{d-1}$ is given by
\begin{equation}\label{Eq7a.2}
H_k = \{\varphi  :\wt\Delta  \varphi  = -k(k+d-2)\varphi  \}, \q
\lambda  _k \equiv k(k+d-2).
\end{equation}
For a Banach space of functions on $S^{d-1}$ the $K\text{\rm
-functional} \; \wt K_r\big(f,(-\wt \Delta  )^{r/2},t^r\big)_B$ is
given by
\begin{equation}\label{Eq7a.3}
K_r\big(f,(-\wt\Delta  \big)^{r/2},t^r\big)_B = \,\text{\rm inf}\,
\big\{\Vert  f-g\Vert  _B + t^{r}\Vert  (-\wt \Delta  )^{r/2}g\Vert
_B:g\in \SD\big((-\wt \Delta  )^{r/2}\big)\big\}
\end{equation}
where
\begin{equation}\label{Eq7a.4}
\begin{gathered}
(-\wt \Delta  )^{r/2}f \sim \sum^\infty  _{\ell=1}\big(\ell (\ell
+d-2)\big)^{r/2}P_\ell f,\,\\ f\in \SD\big((-\wt \Delta
)^{r/2}\big)\q
\text{\rm if} \q \exists\;\psi\in B\q\text{\rm
satisfying}\q P_\ell \psi = \lambda  ^{r/2}_\ell P_\ell f
\end{gathered}
\end{equation}
and $P_\ell \varphi  $ is the $L_2$ projection of $f$ on $H_\ell\,.$

We can now state and prove the result of this section.

\begin{theorem}\label{Thm7a.1}
Suppose that $B$ is an Orlicz space of functions on $S^{d-1}$ satisfying
{\rm (\ref{Eq1.2})} for some $2\le s\le\infty  ,$ and for $\rho\in
SO(d)$
\begin{equation}\label{Eq7.5n}
\Vert  f(\rho \,\cdot\,)\Vert  _B = \Vert  f(\,\cdot\,)\Vert  _B, \q
\Vert  f(\rho\,\cdot\,) - f(\,\cdot\,)\Vert  _B \to 0\q\text{\rm
as}\q \vert  \rho - I\vert  \to 0,
\end{equation}
where $|\rho - I|=\max\{|\rho \bx-\bx|: \bx\in S^{d-1}\}$.
Then for $r=1,2,\dots$
\begin{equation}\label{Eq7a.5}
K_{r}\big(f,(-\wt\Delta  )^{r/2},t^r\big)_B\ge C\Big\{\sum^\infty_{j=1}
\,2^{-jrs}K_{r+1} \big(f,(-\wt \Delta
)^{(r+1)/2},t^{r+1}2^{j(r+1)}\big)^s_B\Big\}^{1/s}
\end{equation}
and
\begin{equation}\label{Eq7a.6}
K_{r}\big(f,(-\wt \Delta  )^{r/2},2^{-nr}\big)_B \ge
C_1\,2^{-nr}\Big\{\sum^n_{j=0} 2^{j rs}E_{2^j}(f)^s_B\Big\}^{1/s}
\end{equation}
where $\wt \Delta  ,$ \, $K_{r}\big(f,(-\wt \Delta
)^{r/2},t^r\big)_B$and $(-\wt \Delta  )^{r/2}$ are given by
{\rm (\ref{Eq7a.1}), (\ref{Eq7a.3})} and {\rm (\ref{Eq7a.4})} respectively, and
$E_n(f)_B$ is given by
\begin{equation}\label{Eq7a.7}
E_n(f)_B = \;\text{\rm inf}\;\Big(\Vert  f-P\Vert  _B: P \in \;\text{\rm span}\; \bigcup^{n-1}_{k=0} H_k\Big)
\end{equation}
with $H_k$ of \rm{(\ref{Eq7a.2})}.
\end{theorem}

We remind the reader that $SO(d)$ is the collection of $d\times d$ orthogonal
matrices whose determinant equals $1.$

\begin{proof}
We first recall that the Ces\`aro summability of order $\ell >d-1$
is a positive operator (see for instance
\cite[Cor.7.2.5,p.266]{Du-Xu}).  This already implies that $C^\ell_n$
is a contraction operator on $L_p(S^{d-1}).$  Furthermore, the above
and \cite[Th.2.1]{Di06} imply that $C^\ell_n$ is a contraction on
many other Banach spaces of functions on $S^{d-1},$ including all
Orlicz spaces.  We now use the semigroup given in (\ref{Eq5.3}) and
Theorem \ref{Thm5.2} to obtain (\ref{Eq7a.5}) when we observe that,
using the technique of \cite[Section 4]{Da-Di07},
$K_{\CA^r}(f,t^r)_B\approx K_r\big(f,(-\wt \Delta  )^{r/2},t^r\big)_B$ for
that semigroup for any Banach space $B$ for which the Ces\`aro
summability is bounded.  The inequality (\ref{Eq7a.6}) follows using
\cite[Th.6.1]{Di98}, which is applicable here as the Ces\`aro
summability is bounded and implies
\begin{equation}\label{Eq7a.8}
E_{2^k}(f)_B \le CK_{r+1}\big(f,(-\wt\Delta  )^{(r+1)/2},2^{-k(r+1)}\big)_B.
\end{equation}
\end{proof}

\section{Non-holomorphic semigroups and averaged moduli of
smoothness}\label{Sec7}
For a semigroup $\{T(u)\}_{u\ge 0}$ on a Banach space $B$ the
averaged moduli of smoothness are given by
\begin{equation}\label{Eq7.1}
\text{\rm w}^r_T(f,t)_B \equiv \frac 1t\;\int^t_0\,\Vert  \big(T(u) -
I\big)^rf\Vert  _Bdu.
\end{equation}
We recall that the moduli $\omega  ^r_T(f,t)_B$ are given by
\begin{equation}\label{Eq7.2}
\omega  ^r_T(f,t) _B\equiv \us{0\le u\le t}\sup\, \big\Vert
\big(T(u)-I\big)^rf\big\Vert  _B\,,
\end{equation}
and we have the following equivalence.

\begin{theorem}\label{Thm7.1}
Suppose $\{T(u)\}_{u\ge 0}$ is a $C_0$ semigroup of contractions on
a Banach space $B.$  Then
\begin{equation}\label{Eq7.3}
\text{\rm w}^r_T(f,t)_B\le \omega  ^r_T(f,t)_B \le C(r)\text{\rm w}^r_T(f,t)_B.
\end{equation}
\end{theorem}

\begin{proof} We now follow verbatim the proof in \cite[p.184-185]{De-Lo}.  In \cite{De-Lo} the result refers only to $L_p$
and translations, but the proof is the same and the identity
(\ref{Eq5.3}) in \cite[p.184]{De-Lo} is replaced by the identity
\begin{equation}\label{Eq7.4}
\big(T(h) - I\big)^r
=\sum^r_{k=1}(-1)^k\begin{pmatrix}r\\k\end{pmatrix}
\Big\{T(kh)\big(T(ks)-I\big)^r - \big(T(h+ks)-I\big)^r\Big\},
\end{equation}
the proof of which is the same.
\end{proof}

As a corollary, we obtain the following result.

\begin{theorem}\label{Thm7.2}
Suppose $\{T(u)\}_{u\ge 0}$ is a $C_0$ semigroup of contractions on
a Banach space $B$ which satisfies {\rm (\ref{Eq1.2})} for some $2\le
s<\infty  .$  Then
\begin{equation}\label{Eq7.4}
\omega  ^r_T(f,t)_B\ge C\Big\{\sum^\infty_{j=1} 2^{-jrs} \omega
^{r+1}_T(f,2^jt)^s_B\Big\}^{1/s}.
\end{equation}
\end{theorem}

\begin{proof}
Using the definition of $\omega  ^r_T(f,t)_B,$ we have
$$
\omega  ^r_T(f,t)^s_B\ge \frac 1t\,\int^t_0\, \big\Vert
\big(T(u)-I\big)^rf\big\Vert  ^s_Bdu.
$$

Theorem \ref{Thm2.1} now implies
$$
\omega  ^r_T(f,t)^s_B \ge \frac{m_1}{t}\,\int^t_0\;\sum^\ell_{j=0}
\, 2^{-rsj} \big\Vert  \big(T(2^ju)-I\big)^{r+1}f\big\Vert  ^s_B du
$$
(which setting $v= 2^ju)$
\vspace{-10pt}
$$
\begin{aligned}
&= m_1\;\sum^\ell_{j=0}\, 2^{-rsj}\, \frac{1}{2^jt}\;\int^{2^jt}_0
\, \big\Vert  \big(T(v)-I\big)^{r+1}f\big\Vert  ^s_Bdv\\
{}\q\q\q\q\q\q\q &\ge m_1\;\sum^\ell_{j=0}\, 2^{-rsj}
\Big(\,\frac{1}{2^jt}\;\int^{2^jt}_0 \,\big\Vert
\big(T(v)-I\big)^{r+1}f\big\Vert  _Bdv\Big)^s\\
&= m_1\;\sum^\ell_{j=0}\, 2^{-rsj}\, {\text{\rm w}}^{r+1}_T (f,2^jt)^s_B\\
&\ge \frac{m_1}{\big(C(r+1)\big)^s} \,\sum^\ell_{j=0}\, 2^{-rsj} \omega
^{r+1}_T (f,2^jt)^s_B\,.
\end{aligned}
$$
~\end{proof}

As an immediate application, we obtain the following result.

\begin{theorem}\label{Thm7.3}
Suppose  $B$ is a Banach space of functions on $\IR_+,\,\IR$ or $\IT$
satisfying {\rm (\ref{Eq1.2})} with some $2\le s<\infty  $ and $\Vert
f(\cdot \,+\xi  )\Vert  _B\le \Vert  f(\,\cdot\,)\Vert  _B$ for $\xi
 \ge 0.$ Then
\begin{equation}\label{Eq7.5}
\omega  ^r(f,t)_B \ge C\Big\{\sum^\infty_{j=1} \, 2^{-jrs} \omega
^{r+1}(f,2^jt)^s_B\Big\}^{1/s}
\end{equation}
where $\omega  ^k(f,t)_B$ is $\omega  ^k_T(f,t)_B$ with $T(u)f(x) =
f(x+u).$
\end{theorem}

We remark that for $\IR_+\,,$ Theorem \ref{Thm7.3} was not deduced in
\cite{Da-Di-Ti} even for $L_p(\IR_+),$ with $1<p<\infty .$  Of course
(\ref{Eq7.5}) is valid for other spaces, not just $L_p.$

\section{Results for spaces of functions on $\IR^d$ or $\IT^d,\; d>1$}
\label{Sec8}
For $d>1$ we use a result on averaged moduli that stems from the work
\cite{Da-Di04} which is different from the averaged moduli in Section
\ref{Sec7}.

We define
\begin{equation}\label{Eq8.1}
V_tf(x) = \frac{1}{m_t} \;\int_{\vert  x-y\vert  =t} \, f(y)dy, \q
V_t1=1
\end{equation}
where $\vert  x-y\vert  $ is the Euclidean distance between $x$ and
$y$
for which we have the following result.

\begin{theorem}\label{Thm8.1}
Suppose $B$ is a Banach space of functions on $\IR^d$ or $\IT^d$ with
$d>1$ which satisfies {\rm (\ref{Eq1.3})}. Then
\begin{equation}\label{Eq8.2}
\Vert  V_{\ell,t}f-f\Vert  _B\approx \us g\inf\, (\Vert  f-g\Vert
_B + t^{2\ell}\Vert  \Delta  ^\ell g\Vert  _B) \equiv K_{\Delta
^\ell}(f,t^{2\ell})_B
\end{equation}
where
\begin{equation}\label{Eq8.3}
V_{\ell,t}f \equiv
\,\frac{-2}{\begin{binom}{2\ell}{\ell}\end{binom}}\;\sum^\ell_{j=1}\,
(-1)^j\begin{pmatrix} 2\ell\\ \ell-j\end{pmatrix} V_{jt}f
\end{equation}
and $\Delta f\equiv\frac{\partial^2f}{\partial x_1^2}+\dots+\frac{\partial^2f}{\partial x_d^2}$ is the Laplacian.
\end{theorem}

\begin{proof}
For $L_p(\IR^d)$\; $1\le p\le \infty  $ and $d>1,$ Theorem \ref{Thm8.1}
was proved in \cite[Th.3.1, pp.273-276]{Da-Di04}, and in fact all we
do here is show how to deduce our theorem from \cite[Th.3.1]{Da-Di04}.
We note that (\ref{Eq8.2}) for $L_\infty  (\IR^d)$
implies the validity of (\ref{Eq8.2}) for $C(\IR^d).$  (Perhaps the
only interesting situation of (\ref{Eq8.2}) in case $B= L_\infty
(\IR^d)$ is when $B= C(\IR^d),$ because only when $f\in C(\IR^d)$ do both
sides of (\ref{Eq8.2}) tend to zero as $t\to 0.)$

Using \cite[Th.6.2,p.97]{Be-Da-Di} with
$m_1>\frac{2(d+2)}{d-1}\,\ell,$ we have
$$
\Vert  \Delta  ^\ell V^{m_1}_{kt}F\Vert  _{C(\IR^d)} \le
\frac{A_1(m_1,\ell,k)}{t^{2\ell}}\, \Vert  F\Vert  _{C(\IR^d)},
$$
and hence for $m$ large enough, $m>\frac{2(d+2)}{d-1}\,\ell^2$ for
example, we have
\begin{equation}\label{Eq9.4}
\Vert  \Delta  ^\ell V^m_{\ell,t} F\Vert  _{C(\IR^d)} \le
\frac{A_2(m,\ell)}{t^{2\ell}} \,\Vert  F\Vert  _{C(\IR^d)}\,.
\end{equation}
We now show that for $F\in C(\IR^d)$
\begin{equation}\label{Eq9.5}
\begin{aligned}
A^{-1}\Vert F-V_{\ell,t}F\Vert  _{C(\IR^d)}
&\le \Vert  F-V^m_{\ell,t}F\Vert  _{C(\IR^d)} + t^{2\ell}\Vert
\Delta  ^\ell V^m_{\ell,t}F\Vert  _{C(\IR^d)}\\
&\le A\Vert  F-V_{\ell,t}F\Vert  _{C(\IR^d)}\,.
\end{aligned}
\end{equation}

The left hand inequality of (\ref{Eq9.5}) is clear using
\eqref{Eq8.2} for $C(\IR^d)$ (already proved in
\cite[Th.3.1]{Da-Di04}), and recalling the definition of $K_{\Delta
^\ell}(f,t^{2\ell})_{C(\IR^d)}.$  Using $\Vert  V_{\ell,t}F\Vert
_{C(\IR^d)} \le A_1\Vert  F\Vert  _{C(\IR^d)}\,,$ we have
$$
\Vert  F-V^m_{\ell,t}F\Vert  _{C(\IR^d)} \le A_2\Vert
F-V_{\ell,t}F\Vert  _{C(\IR^d)}\,.
$$

To conclude the proof of \eqref{Eq9.5} we have to estimate
$t^{2\ell}\Vert  \Delta  ^\ell V^m_{\ell,t}F\Vert  _{C(\IR^d)}\,.$
We choose $G_1$ such that $\Vert  F-G_1\Vert  _{C(\IR^d)} +
t^{2\ell}\Vert  \Delta  ^\ell G_1\Vert  _{C(\IR^d)} \le 2K_{\Delta
^\ell}(F,t^{2\ell})_{C(\IR^d)}$ and write
$$
\begin{aligned}
t^{2\ell}\Vert  \Delta  ^\ell V^m_{\ell,t}F\Vert  _{C(\IR^d)}
&\le t^{2\ell} \Vert  \Delta  ^\ell V^m_{\ell,t}(F-G_1)\Vert
_{C(\IR^d)} + t^{2\ell}\Vert  \Delta  ^\ell V^m_{\ell,t}G_1\Vert
_{C(\IR^d)}\\
&\le A_2 (m,\ell)\Vert  F-G_1\Vert  _{C(\IR^d)} + t^{2\ell}\Vert
V^m_{\ell,t}\Delta  ^\ell G_1\Vert  _{C(\IR^d)}\\
&\le A_2(m,\ell)\Vert  F-G_1\Vert  _{C(\IR^d)} + t^{2\ell}A^m_3
\Vert  \Delta  ^\ell G_1\Vert  _{C(\IR^d)}\\
&\le A_4 K_{\Delta  ^\ell}(F,t^{2\ell})_{C(\IR^d)} \\
&\le A_5 \Vert  F-V_{\ell,t}F\Vert  _{C(\IR^d)},
\end{aligned}
$$
which concludes the proof of \eqref{Eq9.5}.  To prove \eqref{Eq8.2}
for a Banach space on $\IR^d$ or $\IT^d,$ we proceed first by showing
\begin{equation}\label{Eq9.6}
\begin{aligned}
A^{-1} \Vert  f- V_{\ell,t}f\Vert  _B
&\le \Vert  f-V^m_{\ell,t}f\Vert  _B + t^{2\ell}\Vert  \Delta  ^\ell
V^m_{\ell,t}f\Vert  _B\\
&\le 2A \Vert  f-V_{\ell,t}f\Vert  _B\,.
\end{aligned}
\end{equation}
We first attend to Banach spaces $B$ of functions on $\IR^d.$  To
prove the left hand inequality of \eqref{Eq9.6}, we choose $g\in B^*$
satisfying $\Vert  g\Vert  _{B^*}=1$ and define $F(x) = f*g(x) =\la
f(x-\,\cdot\,),g(\,\cdot\,)\ra\,.$
Using \eqref{Eq1.3} we have $F\in C(\IR^d)$ and recalling
\eqref{Eq9.5}, we have
$$
A^{-1}\Vert  F-V_{\ell,t}F\Vert  _{C(\IR^d)} \le \Vert
F-V^m_{\ell,t} F\Vert  _{C(\IR^d)} + t^{2\ell}\Vert  \Delta  ^\ell
V_{\ell,t}F\Vert  _{C(\IR^d)}
$$
(so using $\Vert  g\Vert  _{B^*}=1$ and the convolution structure of
$V_{\ell,t}$ will imply)
$$
\le \Vert  f-V^m_{\ell,t}f\Vert  _B + t^{2\ell}\Vert  \Delta  ^\ell
V_{\ell,t}f\Vert  _B\,.
$$
For appropriate $g_\varepsilon  $ and $F=F_\varepsilon
=f*g_\varepsilon  $
$$
\Vert  F-V_{\ell,t}F\Vert  _{C(\IR^d)} \ge \vert  F(0) -
V_{\ell,t}F(0)\vert  \ge \Vert  f-V_{\ell,t}f\Vert_B  -\varepsilon  ,
$$
and as $\varepsilon  >0$ is arbitrary, the left inequality of
\eqref{Eq9.6} is proved.

We now follow the same technique to deduce from
$$
\Vert
F-V^m_{\ell,t}F\Vert  _{C(\IR^d)} \le A\Vert  F-V_{\ell,t}F\Vert
_{C(\IR^d)}\q\text{\rm and}\q t^{2\ell}\Vert  \Delta  ^\ell V^m_{\ell,t}F\Vert
_{C(\IR^d)} \le A\Vert  F-V_{\ell,t}F\Vert  _{C(\IR^d)}
$$
the inequalities
$$
\Vert  f-V^m_{\ell,t}f\Vert  _B\le A\Vert  f-V_{\ell,t}f\Vert
_B\q\text{\rm and}\q t^{2\ell}\Vert  \Delta  ^\ell
V^m_{\ell,t}f\Vert  _B\le A\Vert  f-V_{\ell,t}f\Vert  _B\,,
$$
which together with the above, imply \eqref{Eq9.6} and hence \eqref{Eq8.2} for a Banach
space of functions on $\IR^d$ satisfying \eqref{Eq1.3}.

To prove
the result for a Banach space of functions satisfying \eqref{Eq1.3}
on $\IT^d,$ we observe that $C(\IT^d)\subset C(\IR^d)$ and that for
$F\in C(\IT^d),$ $V^k_{\ell,t}F\in C(\IT^d)$ for all $k,\ell$ and $t.$

Moreover, \eqref{Eq9.5} is satisfied with the norm $C(\IT^d)$
replacing $C(\IR^d),$ since if
\newline$G\in C(\IT^d),$ $\Vert  G\Vert
_{C(\IT^d)} = \Vert  G\Vert  _{C(\IR^d)}\,.$ We now use the same
technique to deduce \eqref{Eq9.6} for Banach spaces of functions on
$\IT^d$ from \eqref{Eq9.5} with $C(\IT^d)$ instead of $C(\IR^d).$

To show that the inequality \eqref{Eq9.6} implies \eqref{Eq8.2}, we
observe that the right hand inequality implies $\Vert
f-V_{\ell,t}f\Vert  _B\ge \frac{1}{2A}\, K_{\Delta
^\ell}(f,t^{2\ell})_B\,.$  Choosing $g$ such that $\Delta  ^\ell
g\in B$ and
\newline $\Vert  f-g\Vert  _B +t^{2\ell}\Vert  \Delta  ^\ell
g\Vert  _B \le 2K_{\Delta  ^\ell}(f,t^{2\ell})_B\,,$ and using the
left inequality of \eqref{Eq9.6}, we write
$$
\begin{aligned}
A^{-1}\Vert  f-V_{\ell,t}f\Vert  _B
&\le \Vert  f-g\Vert  _B + \Vert  V^m_{\ell,t}(f-g)\Vert  _B\\
&\q + t^{2\ell}\Vert  \Delta  ^\ell V^m_{\ell,t}(f-g))\Vert  _B +
t^{2\ell}\Vert  \Delta  ^\ell V^m_{\ell,t}g\Vert  _B.
\end{aligned}
$$

We now follow the method used earlier to deduce
$$
\Vert  V^m_{\ell,t}f\Vert  _B\le A_5\Vert  f\Vert  _B\q\text{\rm
and}\q \Vert  \Delta  ^\ell V^m_{\ell,t} f\Vert  \le
\frac{A_2(m,\ell)}{t^{2\ell}}\, \Vert  f\Vert  _B\q\text{\rm for
all}\q f\in B,
$$
from the corresponding inequalities for $B=C(\IR^d)$ or
$B=C(\IT^d).$  We also need to recall that $\Vert  \Delta  ^\ell
V^m_{\ell,t}g\Vert  _B = \Vert  V^m_{\ell,t}\Delta  ^\ell g\Vert
_B$ whenever $\Delta^\ell g\in B$ to complete the proof.
\end{proof}

\begin{theorem}\label{Thm8.2}
Suppose $B$ is a Banach space of functions on $\IR^d$ or $\IT^d$ and its
norm satisfies {\rm (\ref{Eq1.2})} for some $s,$ \;$2\le s<\infty  ,$
and {\rm (\ref{Eq1.3})}.
  Then for any $\ell$ such that $2\ell >r$
\begin{equation}\label{Eq8.9}
\omega  ^r (f,t)_B\ge C\Big\{\sum^\infty_{j=1} 2^{-jrs}K_{\Delta
^\ell}\big(f,(2^j t)^{2\ell}\big)^s_B\Big\}^{1/s}.
\end{equation}
\end{theorem}

\begin{proof}
We write
$$
\omega  ^r(f,t)^s_B = \us{\vert  u\vert  \le t}\sup\;\Vert  \Delta
^r_u f\Vert  ^s_B \ge \us{\vert  u\vert  =t}\sup\;\Vert  \Delta
^r_u f\Vert  ^s_B\ge \frac{1}{m_t}\;\int_{\vert  u\vert  =t} \Vert
\Delta  ^r_u f\Vert  ^s_B du
$$
with $m_t$ of (\ref{Eq8.1}) i.e. $\int_{\vert  u\vert =t}du =
m_t\,.$  We now use Theorem \ref{Thm2.1} with $T=T(u)$ and $T(u)f(x)
= f(x+u)$ to obtain
$$
\begin{aligned}
\omega  ^r(f,t)^s_B
&\ge \frac{C}{m_t} \; \int_{\vert  u\vert  =t} \;\sum^L_{j=1}
2^{-jrs}\Vert  \Delta  ^{r+1}_{2^ju}f\Vert  ^s_B du\\
&\ge \frac{C_1}{m_t} \;\int_{\vert  u\vert  =
t}\;\sum^L_{j=1}\,2^{-jrs}\Vert  \Delta  ^{2\ell}_{2^ju} f\Vert
^s_B du\\
&= C_1\;\sum^L_{j=1} \, 2^{-jrs} \; \frac{1}{m_t 2^{j(d-1)}}
\;\int_{\vert  v\vert  =2^jt}\;\Vert  \Delta  ^{2\ell}_v f\Vert
^s_B dv.
\end{aligned}
$$
As translations are isometries (see (\ref{Eq1.3})), we have
$$
\Vert  \Delta  ^{2\ell}_v f(\,\cdot\,)\Vert  _B = \Big\Vert
\sum^\ell_{k=-\ell} (-1)^k \begin{binom}{2\ell}{\ell-k}\end{binom}
f(\,\cdot\,+kv)\Big\Vert  _B\,.
$$
Therefore, using the H\"older and the triangle inequality we have
$$
\omega  ^r(f,t)^s_B \ge C_1\;\sum^L_{j=1} 2^{-jrs} \Big\Vert
\frac{1}{m_t 2^{j(d-1)}} \, \int_{\vert  v\vert  =2^jt}
\;\sum^\ell_{k=-\ell}\,(-1)^k\begin{binom}{2\ell}{\ell-k}\end{binom}
f(\,\cdot\,+kv)dv\Big\Vert  ^s_B\,.
$$
Since $\int_{\vert  v\vert  =2^jt}dv = m_t 2^{j(d-1)},$ we now have
(using Theorem \ref{Thm8.1})
$$
\begin{aligned}
\omega  ^r(f,t)^s_B &\ge C_1\,\sum^L_{j=1}
2^{-jrs}\begin{binom}{2\ell}{\ell}\end{binom} \Vert  V_{\ell,2^jt}
f-f\Vert  ^s_B\\
&\ge C_2 \,\sum^L_{j=1} 2^{-jrs} K_{\Delta
^\ell}\big(f,(2^jt)^{2\ell}\big)^s_B\,.
\end{aligned}
$$
\vskip-.3in
\end{proof}

The sharp-Jackson result can now be deduced from Theorem
\ref{Thm8.2}.

\begin{theorem}\label{Thm8.3}
Suppose $B$ is a Banach space of functions on $\IR^d$ or $\IT^d$
satisfying {\rm (\ref{Eq1.2})} and
{\rm (\ref{Eq1.3})}.
  Then
\begin{equation}\label{Eq8.10}
\omega  ^r(f,t)_B \ge C\Big\{\sum^\infty_{j=1} 2^{-jrs}
E_{1/(t2^j)}(f)^s_B\Big\}^{1/s}
\end{equation}
where $E_\lambda  (f)_B$ is given in {\rm (\ref{Eq4.10})}  when
$B$ is a space of functions on $\IR^d$ and by
\begin{equation}\label{Eq8.11}
E_\lambda  (f)_B =\,\inf\,\Big\{\Vert  f-\varphi  \Vert  _B :\,
\varphi  (\bx) = \sum_{\vert  \pmb{n}\vert  <\lambda  } a_{\pmb n} e^{i{\pmb
n}\bx}\Big\}
\end{equation}
when $B$ is a space of functions on $\IT^d.$
\end{theorem}

\begin{proof}
When $E_\lambda  (f)_B$ is given by (\ref{Eq4.10}), we use
(\ref{Eq4.12}) to deduce (\ref{Eq9.4}) from (\ref{Eq8.3}), writing
$$
f = f-\varphi _ {1/t} + (\varphi  _{1/2t} - \varphi  _{1/t})+ \cdots
+ (\varphi  _{1/2^it} - \varphi  _{1/2^{i-1}t}) + \varphi  _{1/2^it}
$$
where $\varphi  _\lambda  $ is a near best approximant i.e. $\Vert
f  -\varphi  _\lambda  \Vert  _B\le aE_\lambda  (f)_B\,.$
When $E_\lambda  $ is given by (\ref{Eq8.11}), we use the analogue
of (\ref{Eq4.12}) and the same expansion to obtain (\ref{Eq8.10}).
\end{proof}

The lower estimate of $\omega  ^r(f,t)_B$ is given in the following
theorem.

\begin{theorem}\label{Thm8.4}
Under the conditions of Theorem {\rm \ref{Thm8.3}}, we have
\begin{equation}\label{Eq8.12}
\omega  ^r(f,t)^s_B  \ge C_1\,\sum^L_{j=1}
2^{-jrs} \omega  ^{r+1}(f,t2^j)^s_B
\end{equation}
{where}$ L= \;\min(\ell:\, 2^{-\ell} \le
t)$
and $B$ is a space of functions on $\IT^d$ or $\IR^d.$
\end{theorem}

\begin{proof}
Since when $2^{-\ell}\le t<2^{-\ell+1},$\; $\omega  ^k(f,2^{-\ell}) \le
\omega  ^k(f,t)_B \le \omega  ^k(f,2^{-\ell+1})_B \le 2^k\omega
^k(f,2^{-\ell})_B,$ it is sufficient to prove (\ref{Eq8.12})
 for $t=2^{-n}$ and $L=n.$  For a Banach space of
functions on $\IR^d$ or $\IT^d$ satisfying (\ref{Eq1.3}), the weak
converse inequality yields
\begin{equation}\label{Eq8.13}
\omega  ^{r+1}(f,2^{-n+j})_B\le C_2\Big\{\sum^{n-j}_{k=0} \,
2^{-k(r+1)} E_{2^{n-j-k}}(f)_B + \frac{1}{2^{(n-j)(r+1)}} \,\Vert
f\Vert  _B\Big\}\,.
\end{equation}
Therefore, recalling $2\le s<\infty  ,$ we have
$$
\begin{aligned}
\sum^n_{j=1} 2^{-jrs} \omega  ^{r+1}(f,2^{-n+j})^s_B
&\le C^s_2\Big\{\sum^n_{j=1} 2^{-jrs}\Big(\sum^{n-j}_{k=0}
2^{-k(r+1)}E_{2^{n-j-k}}(f)_B
 + \frac{1}{2^{(n-j)(r+1)}}\,\Vert  f\Vert  _B\Big)^s\Big\}\\
&\le C_3 \Big\{\sum^n_{j=1} 2^{-jrs}\Big(\sum^{n-j}_{k=0}
2^{-ks(r+1)}E_{2^{n-j-k}}(f)^s_B\Big)\\
&\q +
\sum^n_{j=1} 2^{-jrs} 2^{-(n-j)(r+1)s}\Vert  f\Vert
^s_B\Big\}\\
&\le C_3\Big[\sum^n_{j=1}
\,2^{-jrs}\,\sum^{n-j}_{m=0}\,2^{-(n-j-m)s(r+1)}E_{2^m}(f)^s_B\Big]
+ C_3 2^{-nrs}\Vert  f\Vert  ^s_B\\
&= C_3\Big[\sum^n_{m=0} E_{2^m}(f)^s_B
2^{-(n-m)s(r+1)}\,\sum^{n-m}_{j=1} 2^{js}\Big] + C_3 2^{-nrs}\Vert
f\Vert  ^s_B\\
&\le C_4\,\sum^n_{m=1} E_{2^m}(f)^s_B 2^{-(n-m)sr} + C_3
2^{-nrs}\Vert  f\Vert  ^s_B\,.
\end{aligned}
$$

In view of (\ref{Eq8.10}) (for $t=2^{-n}),$ we have
\begin{equation}\label{Eq9.13}
\omega  ^r(f,t)^s_B + t^{rs}\Vert  f\Vert  ^s_B \ge C_5 \sum^L_{j=1}
2^{-jrs} \omega  ^{r+1}(f,t2^j)^s_B\,.
\end{equation}

We choose $g$ so that $\Vert  f-g\Vert  _B = E_1(f)_B$ where
$E_\lambda  (f)_B$ is given in \eqref{Eq4.10} and \eqref{Eq8.13} for
function spaces on $\IR^d$ or $\IT^d$ respectively. Using
\eqref{Eq9.13}, we now write
$$
\begin{aligned}
\sum^L_{j=1} 2^{-jrs} \omega  ^{r+1} (f-g,t2^j)^s_B
&< C^{-1}_5 \{\omega  ^r(f-g,t)^s_B + t^{rs} \Vert  f-g)\Vert^s  _B\}\\
&\le C_6 \omega  ^r(f,t)_B
\end{aligned}
$$
since $\Vert  f-g\Vert  _B + C\omega  ^r(f,1)_B \le C_1 t^{-r}\omega
 ^r(f,t)_B$ (see \cite[Th.2.1]{Di99}).

For $g\in C^r$ one has
\begin{equation}\label{Eq9.13a}
\omega^r(g,\tau)_B\le\tau^r\max_\xi\Big\|\Big(\frac\partial{\partial\xi}\Big)^rg\Big\|_B.
\end{equation}
This follows from $\displaystyle \|\Delta_h^r g\|_\infty\le |h|^r \Big\|\Big(\frac\partial{\partial\xi}\Big)^rg\Big\|_\infty$ for $h$ in the $\xi$ direction and hence following the arguments used in Theorem~\ref{Thm8.1} (and elsewhere), $\displaystyle \|\Delta_h^r g\|_B\le |h|^r \Big\|\Big(\frac\partial{\partial\xi}\Big)^rg\Big\|_B$ (with $h$ still in the $\xi$ direction). Using~\eqref{Eq1.5} we now have~\eqref{Eq9.13a}.

Therefore, as $g\in C^\infty  $ and $2^{-L}\approx t,$ we have
$$
\begin{aligned}
\sum^L_{j=1} 2^{-jrs}\omega  ^{r+1}(g,t2^j)^s_B
&\le \sum^L_{j=1} 2^{-jrs}(t2^j)^{(r+1)s}\, \us\xi  \max\, \Big\Vert
\Big(\pd{}{\xi  }\Big)^{r+1}g\Big\Vert  ^s_B\\
&\le C_7 t^{rs} \,\us{\vert  \xi\vert  =1} \max\, \Big\Vert  \Big(\pd{}{\xi
}\Big)^{r+1}g\Big\Vert  ^s_B\,.
\end{aligned}
$$
For function spaces on $\IT^d,$ $\big\Vert  \big(\pd{}{\xi
}\big)^{r+1}g\big\Vert  ^s_B =0$. For function spaces on $\IR^d$ we note
that $\supp \, \hat g(y)\subset\{y:|y|\le1\}$ implies
$\supp \, \widehat{(\frac\partial{\partial\xi})^rg}(y)\subset\{y:|y|\le1\}$ and using~\cite[Th.~2.1]{Da-Di04} with
$R=1$ and $\ell=1$, we have
\[
\Big\|\Delta\Big(\frac\partial{\partial\xi}\Big)^rg\Big\|_\infty
\le C \Big\|\Big(\frac\partial{\partial\xi}\Big)^rg\Big\|_\infty
\]
and hence
\[
\Big\|\Delta\Big(\frac\partial{\partial\xi}\Big)^rg\Big\|_B
\le C \Big\|\Big(\frac\partial{\partial\xi}\Big)^rg\Big\|_B.
\]
We now use~\cite[Th.~6.2]{Di89} to obtain
\[
\Big\|\Big(\frac\partial{\partial\xi}\Big)^{r+1}g\Big\|_B
\le C_8 \Big\|\Delta\Big(\frac\partial{\partial\xi}\Big)^rg\Big\|_B^{1/2}
\Big\|\Big(\frac\partial{\partial\xi}\Big)^rg\Big\|_B^{1/2}
\le C_9 \Big\|\Big(\frac\partial{\partial\xi}\Big)^rg\Big\|_B.
\]
Therefore,
$$
t^{rs}\,\us{\vert  \xi  \vert  =1}\max\,\Big\Vert  \Big(\pd{}{\xi  }\Big)^{r+1}g\Big\Vert  ^s_B
\le t^{rs}\,\us{\vert  \xi  \vert  =1}\max\,\Big\Vert  \Big(\pd{}{\xi  }\Big)^r g\Big\Vert
^s_B
\le C_{10}\, t^{rs} \omega  ^r(f,1)^s_B \le C_{11}\,\omega  ^r(f,t)^s_B.
$$
\end{proof}

We thank F. Dai for some valuable comments and for showing that the
second term on the left of \eqref{Eq9.13} is redundant not only for
function spaces on $\IT^d.$


\begin{thebibliography}{999999}
\baselineskip12pt
\parskip10pt

{\small

\bibitem[Be-Da-Di]{Be-Da-Di} E. Belinsky, F. Dai and Z. Ditzian,
{\it Multivariate approximating averages}, Jour. Approx. Theory {\bf
125} (2003), 85-105.

\bibitem[Be-Sh]{Be-Sh} C. Bennet and R. Sharpley, {\it Interpolation
of Operators}, Academic Press, 1988.

\bibitem[Bu-Be]{Bu-Be} P.L. Butzer and H. Berens, {\it Semi-groups
of Operators and Approximation}, Springer Verlag, 1967.

\bibitem[Da-Di,04]{Da-Di04} F. Dai and Z. Ditzian, {\it Combinations of
multivariate averages},  Jour. Approx. Theory {\bf 131} (2004),
268-283.

\bibitem[Da-Di,05]{Da-Di05} F. Dai and Z. Ditzian, {\it Strong
converse inequality for Poisson sums}, Proc. Amer. Math. Soc. {\bf
133} (2005), 2609-2611.

\bibitem[Da-Di,07]{Da-Di07} F. Dai and Z. Ditzian, {\it Ces\`aro
summability and Marchaud inequality}, Constr. Approx. {\bf 25}
(2007), 73-88.

\bibitem[Da-Di,08]{Da-Di08} F. Dai and Z. Ditzian, {\it Jackson
inequality for Banach spaces on the sphere}, Acta Math. Hungar. {\bf
118} (2008), 171-195.

\bibitem[Da-Di-Ti]{Da-Di-Ti} F. Dai, Z. Ditzian and S. Tikhonov,
{\it Sharp Jackson inequalities}, Jour. Approx. Theory {\bf 151}
(2008), 86-112.

\bibitem[De-Lo]{De-Lo} R. DeVore and G. Lorentz, {\it Constructive
Approximation}, Springer Verlag, 1993.

\bibitem[Di,88]{Di88} Z. Ditzian, {\it On the Marchaud inequality},
Proc. Amer. Math. Soc. {\bf 103} (1988), 198-202.

\bibitem[Di,89]{Di89} Z. Ditzian, {\it Multivariate Landau-Kolmogorov-type
inequality}, Math. Proc. Camb. Phil. Soc. {\bf 105} (1989), 335-350.

\bibitem[Di,95]{Di95} Z. Ditzian, {\it Multidimensional Jacobi-type
Bernstein-Durrmeyer operators}, Acta Sci. Math. (Szeged) {\bf 60} (1995), 225-243.

\bibitem[Di,98]{Di98} Z. Ditzian, {\it Fractional derivatives and
best approximation}, Acta Math. Hungar. {\bf 81}(4) (1998), 323-348.

\bibitem[Di,99]{Di99} Z. Ditzian, {\it A modulus of smoothness
on the unit sphere}, Jour. D'Anal. Math. {\bf 79} (1999), 189-200.

\bibitem[Di,06]{Di06} Z Ditzian, {\it Approximation on Banach spaces of
functions on the sphere}, { Jour. Approx. Theory} {\bf 140}
(2006), 31-45

\bibitem[Di-Iv]{Di-Iv} Z. Ditzian and K.G. Ivanov, {\it Strong
converse inequalities}, Jour. D'Anal. Math. {\bf 61} (1993), 61-111.

\bibitem[Di-Pr]{Di-Pr} Z. Ditzian and A. Prymak, {\it Sharp Marchaud
and converse inequalities in Orlicz spaces}, Proc. Amer. Math. Soc.
{\bf 135} (2007), 1115-1121.

\bibitem[Du-Xu]{Du-Xu} C.F. Dunkl and Y. Xu, {\it Orthogonal
Polynomials of Several Variables}, Cambridge University Press, 2001.

\bibitem[Li-Tz]{Li-Tz} Y. Lindenstrauss and L. Tzafriri, {\it Banach
Spaces}, Vol. II, Springer-Verlag,  1979.

\bibitem[Ra-Re]{Ra-Re} M.M. Rao and Z.D. Ren, {\it Theory of Orlicz
Spaces}, Marcel Dekker, 1991.

\bibitem[To]{To} V. Totik, {\it Sharp converse theorem of $L^p$
polynomial approximation}, Constr. Approx. {\bf 4} (1988), 419-433.
}

\end{thebibliography}
\end{document}